\documentclass[8pt,a4paper]{article}

\usepackage{amsmath,amsfonts,amssymb}
\usepackage{bbm}
\usepackage{cases}

\usepackage{kotex}
\usepackage{lipsum}
\usepackage{subcaption}

\usepackage{xcolor}
\usepackage{graphicx}
\usepackage{ulem}
\usepackage{subcaption}

\newtheorem{defn}{Definition}[section]
\newtheorem{theo}[defn]{Theorem}
\newtheorem{lem}[defn]{Lemma}
\newtheorem{prop}[defn]{Proposition}

\newtheorem{rem}[defn]{Remark}
\newtheorem{exam}[defn]{Example}

\newenvironment{proof}{{\bf Proof }}{{\vskip 0.1cm \hfill$\Box$}}

\begin{document} 

\noindent
{\Large \bf Robust error estimates of PINN in one-dimensional boundary value problems for linear elliptic equations}
\\ \\
\bigskip
\noindent
{\bf Jihahm Yoo, Haesung Lee{\footnote{Corresponding author }}}  \\
\noindent
{\bf Abstract.}  In this paper, we study physics-informed neural networks (PINN) to approximate solutions to one-dimensional boundary value problems for linear elliptic equations and establish robust error estimates of PINN regardless of the quantities of the coefficients. In particular, we rigorously demonstrate the existence and uniqueness of solutions using the Sobolev space theory based on a variational approach. Deriving $L^2$-contraction estimates, we show that the error, defined as the mean square of the differences between the true solution and our trial function at the sample points, is dominated by the training loss. Furthermore, we show that as the quantities of the coefficients for the differential equation increase, the error-to-loss ratio rapidly decreases. Our theoretical and experimental results confirm the robustness of the error regardless of the quantities of the coefficients. 


\section{Introduction}
\subsection{Literature review}
\noindent
Neural network-based methods have recently gained significant attention for solving problems in mathematical sciences, particularly in numerical analysis. There has been active research for neural network-based methods to approximate solutions to various differential equations (see \cite{YYK15} and references therein). For instance, if a true solution to a differential equation is not explicitly known, one can train neural networks by minimizing a training loss associated with satisfying the differential equation, without relying on any actual data points that the solution fulfills. These neural networks trained to satisfy differential equations are called physics-informed neural networks (PINN). \\
The basic idea for PINN may originate from \cite{LK90, DP94}. Particularly, in \cite{DP94} one-layer neural network is used to approximate a solution to a special form of the two-dimensional Poisson equation with homogeneous boundary conditions. In \cite{LLF98}, instead of using a neural network directly, the authors use a trial function which is a slight variation in neural networks forcing to satisfy the initial or boundary conditions.
Indeed, it was discussed in \cite{MZ22, SZK23} that using a trial function forcing to satisfy the boundary condition is superior to using a standard neural network in terms of performance such as acceleration of convergence. Recently in \cite{CRBD18, RPK19}, PINN-based methods have been potent tools for various differential equations solvers with the advent of efficient utilities such as TensorFlow and the help of advanced GPU capability. In particular, \cite{RPK19} demonstrates that the solution of the two-dimensional Navier-Stokes equation and the Burgers equation can be approximated very effectively using PINN. \\
In addition to the success of empirical analysis in PINN, there has been various theoretical research to show the convergence of neural networks to a solution for each case of differential equations. One of the most important ingredients of using neural networks to approximate solutions to differential equations is the universal approximation theorem (see \cite{HSW89, C89, R23}). Although the universal approximation theorem mathematically guarantees that any continuous function on a compact interval can be uniformly approximated by neural networks, the theorem is crucially based on a functional analytic result, Stone-Weierstrass theorem (see \cite[Theorem 7.32]{R76}), making it difficult to find an approximating neural network explicitly. To overcome this limitation, one can use a posteriori analysis based on the stability estimates of PDEs. Indeed, one can show that the error related to the difference between the true solution and the neural network is controlled by the training loss which we desire to minimize. For this type of research, we refer to \cite{BTU22} that
presents theoretical aspects of error estimates for the Navier-Stokes equation through stability estimates based on a functional analytic approach.
We also refer to \cite{SJH23, SDK20}  which importantly use the Schauder estimates and $H^{2}$-estimates for their error analysis, respectively.
We mention \cite{MM22} and \cite{MM23}, which provide error estimates in PINN in the context of inverse and forward problems, respectively, within an abstract framework. Additionally, we address \cite{DM22} and \cite{DJM24}, which rigorously derive error estimates in PINNs, with explicit convergence rates for the Kolmogorov equations and the Navier-Stokes equations, respectively. More recently, \cite{YKGL24} shows error estimates in PINN for initial value problems for ordinary differential equations, and \cite{HKL24} develops a new error analysis in PINN based on classical finite element approximation. In \cite{ZMA24}, a new abstract framework unifying forward and inverse problems in PINN is developed, and $H^{1/2}$-error estimates are obtained by utilizing C\'{e}a's Lemma and the coercivity of bilinear forms. In \cite{BTU22, SJH23, SDK20, MM22, MM23, MZ22, ZMA24}, while the error estimates in PINN show that the error decreases as the loss decreases, the lack of explicit calculation of the constants in these estimates makes it challenging to determine precise upper bounds for the error-to-loss ratio.
\subsection{Idea of the error estimates in PINN}
\noindent
Before explaining the main results in this paper, we first consider the following Dirichlet problem related to a general  linear differential operator $\mathcal{L}$ on a bounded open subset $U$ of $\mathbb{R}^d$ with sufficiently regular boundary $\partial U$:
\begin{equation} \label{underlypde}
\left\{
\begin{alignedat}{3}
\mathcal{L}u  &=&& f &\quad&\;\; \text{ in $U$ }\;\;  \\
u &= && g  \;\; &\quad&\;\; \text{ on $\partial U$},
\end{alignedat} \right.
\end{equation}
where $u$, $f$ and $g$ are included in some sufficiently regular function spaces. Assume that for each $f \in L^2(U)$ and $g \in L^2(\partial U)$ there exists a unique solution $u$ to \eqref{underlypde} such that the following $L^2$-error estimate holds:
$$
\|u\|_{L^2(U)} \leq C_{\mathcal{L}} \left(\|f\|_{L^2(U)} +\|g\|_{L^2{(\partial U)}}  \right),
$$
where $C_{\mathcal{L}}>0$ is a constant independent of $u$, $f$ and $g$. 
Now let $\Phi$ be a smooth trial function (for instance, one can choose $\Phi$ as a neural network). Then, replacing the above $u$, $f$ and $g$ by $\Phi-u$, $\mathcal{L}[\Phi]-f$ and $\Phi-g$, respectively, we get
\begin{equation} \label{contityestim}
\|\Phi-u\|_{L^2(U)} \leq C_{\mathcal{L}} \left( \|\mathcal{L}[\Phi]-f\|_{L^2(U)}+ \|\Phi-g\|_{L^2(\partial U)}  \right).
\end{equation}
Then, as a consequence of Monte Carlo integration (see Proposition \ref{montecar}), we expect to obtain that for sufficiently large $n \in \mathbb{N}$
$$
\underbrace{\frac{1}{n}\sum_{i=1}^n \big |  u(X_i)-\Phi(X_i) \big |^2}_{=:\text{\sf Error}} \leq 2C^2_{\mathcal{L}} \bigg(  \underbrace{\frac{1}{n}\sum_{i=1}^n \big |  \mathcal{L}[\Phi](X_i)-f(X_i) \big |^2+\frac{1}{n}\sum_{i=1}^n \big |  \Phi(Y_i)-g(Y_i) \big |^2}_{\text{\sf =:Loss}} \bigg), \quad \text{ very likely},
$$
where $(X_i)_{i \geq 1}$ and $(Y_i)_{i \geq 1}$ are sequences of independent and identically distributed random variables on a probability space $(\Omega, \mathcal{F}, \mathbb{P})$  that have continuous uniform distributions on $U$ and $\partial U$, respectively. Ultimately, we can expect a reduction of {\sf Error} by attempting to minimize \text{\sf Loss}. \\
Here one needs to observe the behavior of constant value $C_{\mathcal{L}}$. As the quantities of the coefficients of $\mathcal{L}$ increase, one may expect that $C_{\mathcal{L}}$ or \text{\sf Loss} also increases. 
In that case, discrepancies in boundary values between the neural networks and the true solutions are expected to remain. Indeed, we observe in our numerical experiment that the {\sf Boundary Loss} is larger than {\sf Error}, regardless of the coefficient (see Example \ref{pinn12compare}). To overcome boundary mismatch, we can use a trial function $\Psi$ ($\Psi=\Phi_A+\ell$ in our case, see Section \ref{basicexperime}) that exactly satisfies the boundary value 
and can obtain improved error analysis (cf. \cite{MZ22}). 
It has been difficult to find the existing literature that explicitly calculates the constant $C_{\mathcal{L}}$. In many cases, the $L^2$-error estimates are based on functional analytic results, making it difficult to explicitly compute the constant $C_{\mathcal{L}}$ though its existence is guaranteed (cf. \cite{ZMA24}). However, our main results, which will be explained in the next section, provide a concrete value for the constant $C_{\mathcal{L}}$.

\subsection{Main results}
\noindent
Now let us deal with our main problem
\begin{equation} \label{concreprobl}
\left\{
\begin{alignedat}{2}
& -\varepsilon \widetilde{y}''+b \widetilde{y}' +c \widetilde{y}&&= f \quad \text{ on } \, I=(x_1, x_2), \\
&\widetilde{y}(x_1)=p, \quad \widetilde{y}(x_2)&&=q,
\end{alignedat} \right.
\end{equation}
where $\varepsilon>0$, $p, q \in \mathbb{R}$ and $f, b, c \in C(\overline{I})$ with $c \geq \lambda$ on $I$ for some constant $\lambda>0$. We first show rigorously the existence and uniqueness of a weak, strong, and classical solution $\widetilde{y}$ to \eqref{concreprobl} by using the Sobolev space theory with a variational approach.  Under the assumption above, we can obtain some error estimates \eqref{weighterbdd} similarly to \eqref{contityestim} (see Theorem \ref{contracproweig}), but the constant value $C_{\mathcal{L}}$ highly depends on $\varepsilon>0$ and the coefficients $b$, in fact, $C_{\mathcal{L}}= \frac{1}{\lambda} \exp\left(\int_{I} \frac{1}{2\varepsilon} |b| dx \right)$.
On the other hand, if we additionally assume that for some constant $\gamma >0$
$$
-\frac{1}{2}b' +c \geq \gamma \quad \text{ in\; $I$ }
$$
and  $\Psi$ is a smooth function with $\Psi(x_1)=p$, $\Psi(x_2)=q$, our main result is the following error estimate (see \eqref{errorestindis} in Section \ref{mainerroranaly}): for sufficiently large $n \in \mathbb{N}$
\begin{align} \label{mainresulestim}
\underbrace{\frac{1}{n} \sum_{i=1}^n |\widetilde{y}(X_i)-\Psi(X_i)|^2}_{=:\text{\sf Error}} \leq \frac{1}{\gamma^2}  \cdot  \, \underbrace{\frac{1}{n} \sum_{i=1}^n \Big(\widetilde{f}(X_i)-\mathcal{L}[\Psi](X_i) \Big)^2}_{=: \text{\sf Loss}}, \quad \text{ very likely}.
\end{align}
Indeed, \eqref{mainresulestim} is an expression in the Monte Carlo integration for the below (see Proposition \ref{montecar} and Theorem \ref{maincontradxes}):
$$
\left( \int_{I} |\widetilde{y} - \Psi|^2 dx \right)^{1/2} \leq  C_{\mathcal{L}}  \left( \int_{I} | \widetilde{f} - \mathcal{L}[\Psi] |^2 dx \right)^{1/2},
$$
where $C_{\mathcal{L}}=\frac{1}{\gamma}$.
From the above estimates, we see that the trial function $\Psi$ is not directly related to a neural network but needs to be a smooth function satisfying the boundary condition, $\Psi(x_1)=p$ and $\Psi(x_2)=q$. But, for our a posteriori analysis, to make {\sf Loss} practically small we will consider the form 
$$
\Psi(x):=-(x-x_1)(x-x_2) N(x; \theta)+\frac{q-p}{x_2-x_1}(x-x_1)+p, \quad x \in I,
$$
where $N(x; \theta)$ is a neural network. The main observation in this paper is that $\frac{\sf Error}{\sf Loss}$ dramatically decreases with the upper bound $\frac{1}{\gamma^2}$ as $\gamma$ increases (see Examples \ref{exampleweight}, \ref{examplefirst}, \ref{example2nd}). For instance, as $\gamma$ increases, {\sf Loss} becomes large, so it is natural to imagine that {\sf Error} may also increase. Actually, we can observe that as $\gamma$ increases, {\sf Loss} tends to increase, but {\sf Error} remains robust since $\frac{\sf Error}{\sf Loss}$, with the upper bound $\frac{1}{\gamma^2}$, rapidly decreases. These observations are mathematically justified by our results where the upper bound of $\frac{\text{\sf Error}}{\text{\sf Loss}}$ is shown to be $\frac{1}{\gamma^2}$.
\\
Another interesting feature in the error estimate \eqref{mainresulestim} is that $\frac{1}{\gamma^2}$ does not depend on $\varepsilon>0$. If $\varepsilon>0$ is very small, then \eqref{concreprobl} is called singularly perturbed convection-diffusion problems, which is known to be quite difficult to solve numerically (see \cite{FH04, RST96, HJ13} and references therein). However, in our estimates \eqref{mainresulestim}, $\frac{\sf Error}{\sf Loss} (\leq \frac{1}{\gamma^2})$ is robust independent of $\varepsilon$, and {\sf Loss} does not increase as $\varepsilon$ decreases to $0$. Thus, if $\gamma$ is fixed, we observe that {\sf Error} also would remain robust independent of $\varepsilon$. These expectations are specifically confirmed in our numerical experiments (see Examples \ref{examplefirst}, \ref{example2nd}).
\subsection{Structure of this paper}
\noindent
This paper is structured as follows. Section \ref{notbasres} presents notations and conventions mainly used in this paper, especially basic results related to one-dimensional Sobolev spaces and probability theory. Section \ref{energysec} derives energy estimates through a variational approach and rigorously shows the existence and uniqueness of solutions. In addition, we will conduct experiments using a standard neural network {\sf (PINN I)} and a trial function that accurately satisfies the boundary conditions {\sf (PINN II)} and will confirm that {\sf (PINN II)} demonstrates superior performance. We will also observe in Example \ref{pinn12newex} that the energy estimate is not optimal in explaining the rapid decrease of the error-to-loss ratio as the quantity of the zero-order term increases.
In Section \ref{mucontraestim}, $L^2(I, \mu)$-contraction estimates and related error estimates are derived, and in the subsequent numerical experiment (Example \ref{weightsimul}), we will analyze the error-to-loss ratio for three $\varepsilon$.
In Section \ref{l2dxcontrasec}, we present our main error estimates \eqref{errorestindis} which improve the error estimates \eqref{weighterbdd} by using $L^2(I, dx)$-contraction estimates. The two subsequent numerical experiments (Examples \ref{examplefirst}, \ref{example2nd}) will effectively support our mathematical analysis.
Section \ref{conout} briefly demonstrates the conclusions and discussion of this paper.

\section{Notations and basic results} \label{notbasres}
\noindent
Here we explain some notations and basic results in this paper. Write $a \vee b:= \max (a,b)$, $a \wedge b:=\min (a,b)$.  Let $I$ be a (possibly unbounded) open interval on $\mathbb{R}$. Let $C(I)$ and $C(\overline{I})$ denote the sets of all continuous functions on $I$ and $\overline{I}$, respectively, and we write $C^0(I):=C(I)$ and $C^0(\overline{I}):=C(\overline{I})$. For each $k \in \mathbb{N} \cup \{ \infty \}$, denote by $C^{k}(I)$ the set of all $k$-times continuously differentiable functions on $I$.
Now let $k \in \mathbb{N} \cup \{0, \infty \}$ and define $C^k_0(I) :=\{f\in C^k(I): \text{$f$ has a compact support in $I$.} \}$ and write $C_0(I):=C_0^0(I)$. By the zero-extension on $\mathbb{R}$, it directly follows that $C_0^k(I)$ is naturally extended to $C^k_0(\mathbb{R})$. Define for each $k \in \mathbb{N} \cup \{\infty\}$
$$
C^{k}(\overline{I}) := \{f \in C(\overline{I}): \text{ there exists a function $\widetilde{f} \in C^k(\mathbb{R})$ such that $\widetilde{f}|_{I}=f$ on $I$}     \}.
$$
Let $\nu$ be a locally finite measure on $\mathbb{R}$. Denote by $dx$ the Lebesgue measure on $\mathbb{R}$ and $dx(I)$ by $|I|$.
For each $r \in [1, \infty)$, denote by $L^r(I, \nu)$ the space of all measurable functions $f$ on $I$ with $\int_{I}|f|^r d\nu<\infty$ equipped with the norm $\|f\|_{L^r(I, \nu)}:= \left(\int_{I}|f|^r d\nu  \right)^{1/r}$. Denote by $L^{\infty}(I, \nu)$ the space of all $\nu$-a.e. bounded measurable functions $f$ on $I$ equipped with the norm 
$$
\|f\|_{L^{\infty}(I, \nu)}:= \inf  \{c>0 :  \nu \left(\{|f|>c \}\right) =0 \}.
$$
Write $L^p(I):= L^p(I, dx)$ for each $p \in [1, \infty]$.
 For a locally integrable function $u$ on $I$, if there exists a locally integrable function $g$ such that $\int_{I} u \varphi' dx = -\int_{I} g \varphi dx $\, for all $\varphi \in C_0^1(I)$, then we write $u'=g$ and we call $u'$ a weak derivative of $u$ on $I$. Let $p \in [1, \infty]$. The Sobolev space $H^{1,p}(I)$ is defined to be
$$
H^{1,p}(I) = \left \{ u \in L^p(I):  \text{ there exists $g \in L^p(I)$ such that $u'=g$ on $I$. } \right\},
$$
equipped with the norm $\|u\|_{H^{1,p}(I)}:=\left(\|u\|^p_{L^p(I)}+ \|u'\|^p_{L^p(I)}\right)^{1/p}$ if $p \in[1, \infty)$ and $\|u\|_{H^{1,\infty}(I)}:=\|u\|_{L^{\infty}(I)}+ \|u'\|_{L^{\infty}(I)}$. Of course, one can also use the notation $W^{1,p}(I)$ for $H^{1,p}(I)$, but here we stick to use the notation $H^{1,p}(I)$. We also define
$$
H^{2,p}(I):=\left \{ u \in H^{1,p}(I):  \text{ there exists $h \in L^p(I)$ such that $(u')'=h$ on $I$. } \right\}.
$$
equipped with the norm $\|u\|_{H^{2,p}(I)}:=\left(\|u\|^p_{L^p(I)}+ \|u'\|^p_{L^p(I)} + \|u'' \|^p_{L^p(I)}\right)^{1/p}$ if\, $p \in[1, \infty)$ and $\|u\|_{H^{2,\infty}(I)}:=\|u\|_{L^{\infty}(I)}+ \|u'\|_{L^{\infty}(I)}+\|u''\|_{L^{\infty}(I)}$. 
An important fact is that $H^{1,p}(I)$ and $H^{2,p}(I)$ are Banach spaces (see Proposition 8.1). For $r \in [1, \infty)$, denote by $H^{1,r}_0(I)$ the closure of $C_0^1(I)$ in $H^{1,r}(I)$.
Let $(\Omega, \mathcal{F}, \mathbb{P})$ be a probability space and $X: \Omega \rightarrow \mathbb{R}$ be a random variable. Define
$$
\mathbb{E}[X] := \int_{\Omega} X d\mathbb{P}, \qquad \; \text{Var}[X]:=\mathbb{E} \left[ \Big( X-\mathbb{E}[X] \Big)^2  \right].
$$
The following well-known results presented in \cite[Theorem 8.2]{Br11} are important in our analysis which describes a connection between weak and classical derivatives. Here we leave the statement for readers.
\begin{prop} \label{basicbrezis}
Let $u \in H^{1,p}(I)$, where $p \in [1, \infty]$ and $I$ is a (possibly unbounded) open interval. Then, the following hold:
\begin{itemize}
\item[(i)]
There exists $\widetilde{u} \in C(\overline{I})$ such that $\widetilde{u}=u$ a.e. on $I$ and that
 \begin{equation*} \label{ftcsobolev}
 \widetilde{u}(y)-\widetilde{u}(x) = \int_x^y u'(t) dt, \quad \text{ for all $x,y \in \overline{I}$.}
 \end{equation*}
\item[(ii)]
If there exists $v \in C(\overline{I})$ such that $u'=v$ a.e. on $I$, then $\widetilde{u}$ as in (i) satisfies $\widetilde{u} \in C^1(\overline{I})$ and that
$$
\lim_{h \rightarrow 0} \frac{\widetilde{u}(x+h)-\widetilde{u}(x)}{h} = v(x), \quad \text{ for all $x \in \overline{I}$.}
$$
(If $x \in \partial \overline{I}$, then the limit above is defined as the one-side limit).
In other words, the classical derivative of \,$\widetilde{u}$  is the same as the weak derivative of $u$ on $I$.
\end{itemize}
\end{prop}
\centerline{}
\noindent
By the Proposition above, every function $u$ in the Sobolev space on an open interval $I$ has a continuous version $\widetilde{u}$ on $\overline{I}$. Therefore, to get a meaning of $u(x)$ for every point $x \in \overline{I}$, we will consider $u$ as its (unique) continuous version $\widetilde{u}$ on $\overline{I}$. The following is also a well-known result (\cite[Theorem 8.12]{Br11}) that characterizes the space $H^{1,r}_0(I)$
\begin{prop}
Let \,$\widetilde{u}\in H^{1,r}(I) \cap C(\overline{I})$ with $r \in [1, \infty)$. Then, $\widetilde{u} \in H^{1,r}_0(I)$ if and only if \, $\widetilde{u}=0$ on $\partial \overline{I}$.
\end{prop}
\centerline{}
\noindent
Below are the basic probabilistic facts about well-known Monte Carlo integration, which will be used in our numerical analysis. We explain it by using the weak law of large numbers. For readers, we present the statement and its proofs.

\begin{prop} \label{montecar}
Let $(X_i)_{i \geq 1}$ be a sequence of independent and identically distributed random variables on a probability space $(\Omega, \mathcal{F}, \mathbb{P})$  that has a continuous uniform distribution on $I$. Let $\Psi \in C(\overline{I})$ with $\Psi \geq 0$ on $\overline{I}$. Then, for each $n, k\in \mathbb{N}$
$$
\mathbb{P} \left( \left| \frac{1}{|I|}\int_{I} \Psi dx -\frac{1}{n}\sum_{i=1}^n \Psi(X_i)   \right| \leq \frac{k \beta}{\sqrt{n}} \right) \geq 1-\frac{1}{k^2},
$$
where $\beta:=\left(\frac{1}{|I|}\int_{I} \Psi^2 dx- \left(\frac{1}{|I|} \int_{I} \Psi dx \right)^2 \right)^{1/2}$.
\end{prop}
\noindent
\begin{proof}
For each $ i \geq 1$, we get $\mathbb{E} \left[ \Psi (X_i)  \right] = \frac{1}{|I|}\int_{I} \Psi dx =: \alpha$ and 
$$
\text{Var}[\Psi(X_i)] =\mathbb{E}\Big[ \Psi(X_i)^2 \Big]  - \mathbb{E}\Big[ \Psi(X_i) \Big]^2  = \frac{1}{|I|}\int_{I} \Psi^2 dx- \left(\frac{1}{|I|} \int_{I} \Psi dx \right)^2 = \beta^2.
$$
Define
$$
\overline{Y}_n:= \frac{1}{n} \sum_{i=1}^n \Psi(X_i).
$$
Note that $\mathbb{E}[\overline{Y}_n]=\frac{1}{n}\sum_{i=1}^n \mathbb{E}[\Psi(X_i)]=\alpha$. Since $\left( \Psi(X_i)  \right)_{i \ge 1}$ is also a sequence of independent and identically distributed random variables on a probability space $(\Omega, \mathcal{F}, \mathbb{P})$, we get 
$$
\text{Var}\big[ \overline{Y}_n \big] = \frac{1}{n^2} \sum_{i=1}^n \text{Var}[\Psi(X_i) ] = \frac{\beta^2}{n}.
$$ 
Using Chebyshev's inequality, for each $c>0$ we get
$$
\mathbb{P} \left(  |\overline{Y}_n-\alpha| \geq c \right) \leq \frac{\beta^2}{c^2n}
$$
By choosing $c:=\frac{k \beta}{\sqrt{n}}$ and considering the complement event, the assertion follows.
\end{proof}
\centerline{}
\noindent
Based on Proposition \ref{montecar}, for sufficiently large but fixed $k \in \mathbb{N}$ if we choose $n$ much larger than $k$, then we can write that
\begin{equation} \label{monteintsim}
\mathbb{P} \left(  \frac{1}{|I|}\int_{I} \Psi dx   \approx  \frac{1}{n}\sum_{i=1}^n \Psi(X_i)     \right)  \approx 1.
\end{equation}
If \eqref{monteintsim} holds, then we will use the following notation that for sufficiently large $n \in \mathbb{N}$
\begin{equation} \label{montecarint}
\frac{1}{|I|}\int_{I} \Psi dx   =  \frac{1}{n}\sum_{i=1}^n \Psi(X_i), \quad \text{very likely},
\end{equation}

\section{Energy estimates} \label{energysec}
\subsection{Existence and uniqueness with energy estimates} 
\begin{lem} \label{embedlem}
Let $u \in H^{1,1}_0(I)$ where $I=(x_1, x_2)$ is a bounded open interval. 
Then, there exists $\widetilde{u} \in C_0(\mathbb{R}) \cap H_0^{1,1}(\mathbb{R})$ with $\widetilde{u}=0$ on $\mathbb{R} \setminus I$ such that $\widetilde{u}=u$ a.e. on $I$ and that
\begin{equation*}
\| u\|_{L^{\infty}(I)}=\|\widetilde{u}\|_{L^{\infty}(\mathbb{R})} \leq \|u'\|_{L^{1}(I)}. 
\end{equation*}
In particular, if $u \in H^{1,2}(I)$, then
$$
|I|^{-1/2}\|u\|_{L^2(I)} \leq \|u\|_{L^{\infty}(I)} \leq |I|^{1/2} \|u'\|_{L^2(I)}.
$$
\end{lem}

\noindent
\begin{proof}
Since $u\in H^{1,1}_0(I)$, there exists a sequence of functions $(u_n)_{n \geq 1} \subset C_0^{\infty}(I)$ such that
$$
\lim_{n \rightarrow \infty} u_n = u, \quad \text{ in \; $H^{1,1}(I)$.}
$$
For each $n \geq 1$, let $\bar{u}_n \in C_0^{\infty}(\mathbb{R})$ be the standard zero-extension of $u_n$ on $\mathbb{R}$. Thus, there exists $\bar{u} \in H_0^{1,1}(\mathbb{R})$ such that $\bar{u}(x)=0$ for a.e. $x \in \mathbb{R} \setminus I$ and $\lim_{n \rightarrow \infty} \bar{u}_n =\bar{u}$ in $H^{1,1}(\mathbb{R})$, and hence $\bar{u}=u$ a.e. on $I$. Meanwhile, by Proposition \ref{basicbrezis}(i) there exists $\widetilde{u} \in C(\mathbb{R})$ such that $\widetilde{u}=\bar{u}$ a.e. on $\mathbb{R}$. Thus, $\widetilde{u} \in C(\mathbb{R}) \cap H^{1,1}(\mathbb{R})$ satisfying that $\widetilde{u}(x)=0$ for all $ x \in \mathbb{R} \setminus I$ and $\widetilde{u}=u$ a.e. on $I$. By the fundamental theorem of calculus (Proposition \ref{basicbrezis}(i)), for each $x \in I$ 
$$ 
\left| \widetilde{u}(x) \right |= \left| \, \widetilde{u}(x_1)+ \int_{x_1}^x \widetilde{u}'(t) dt \right|  = \left| \int_{x_1}^x \widetilde{u}'(t) dt \right| \leq \int_{I} |\widetilde{u}'(t)| dt = \|u'\|_{L^1(I)},
$$
as desired. The rest follows from the H\"{o}lder inequality.
\end{proof}
\centerline{}
\noindent
We are now interested in solving the following Dirichlet boundary value problem for a one-dimensional second-order linear elliptic operator on $I = (x_1, x_2)$:
\begin{equation} \label{undeq1}
\left\{
\begin{alignedat}{2}
& -\varepsilon \widetilde{y}\,''+b \widetilde{y}\,' +c \widetilde{y}&&= \widetilde{f} \quad \text{ on } \, I, \\
&\widetilde{y}(x_1)=p, \quad \widetilde{y}(x_2)&&=q,
\end{alignedat} \right.
\end{equation}
where $p,q \in \mathbb{R}$ and $\varepsilon>0$ are constants, $b \in L^1(I)$ and $c \in L^1(I)$ with $c \geq 0$.
By using the linearity and considering an affine function defined by $\ell(x) = \frac{p-q}{x_2-x_1}(x-x_1)+p$ and replacing $\widetilde{y}$ \,by $y=\widetilde{y}-\ell$ in \eqref{undeq1}, it is enough to study the existence  and uniqueness of solutions to the following homogeneous boundary value problem:
\begin{equation} \label{undeq2}
\left\{
\begin{alignedat}{2}
& -\varepsilon y''+b y' +c y= \widetilde{f}- b \ell' -c\ell \quad \text{ on } \,I, \\
&y(x_1)=0, \quad y(x_2)=0.
\end{alignedat} \right.
\end{equation}
The main idea for solving \eqref{undeq2} is to convert the non-divergence form of \eqref{undeq2} to symmetric divergence form and to apply the Lax-Milgram theorem on the Sobolev space which is $H^{1,2}_0(I)$.

\begin{theo}\rm \label{basictheo}
Let $I=(x_1, x_2)$ be a bounded open interval in $\mathbb{R}$ and let $\varepsilon>0$ be a constant, $b \in L^1(I)$, $c \in L^1(I)$ with $c \geq 0$ and $f \in L^1(I)$. Then, the following hold:

\begin{itemize}
\item[(i)]
Let 
\begin{equation} \label{defnrhofun}
\rho(x):=\exp\left(\int_{x_1}^{x} \frac{-b(s)}{\varepsilon} ds \right), \quad \; x \in \mathbb{R}, 
\end{equation}
where $b$ is considered as the zero-extension of $b \in L^1(I)$ on $\mathbb{R}$. Then, there exists a unique weak solution $y \in H^{1,2}_0(I)\cap C(\overline{I})$ to 
\begin{equation} \label{weaksolu}
-(\varepsilon \rho y')'+\rho c y'= \rho f  \quad \text{ on }\;\; I,
\end{equation}
i.e.
\begin{equation} \label{variequim}
\int_{I} \varepsilon y' \psi' \rho dx + \int_I c y \psi \rho dx = \int_{I} f \psi \rho dx, \quad \text{ for all $\psi \in H^{1,2}_0(I)$}.
\end{equation}
\item[(ii)]
Indeed, $y \in H^{1,2}_0(I)\cap C(\overline{I})$ in (i) satisfies $y \in H^{2,1}(I) \cap C^1(\overline{I})$ and $y$ is a unique strong solution to 
\begin{equation} \label{strongequ}
\left\{
\begin{alignedat}{2}
& -\varepsilon y''+b y' +c y= f \quad \text{ on }  I, \\
&y(x_1)=0, \quad y(x_2)=0.
\end{alignedat} \right.
\end{equation}
i.e.
\begin{equation} \label{strongequdef}
(-\varepsilon y''+ by'+cy)(x) = f(x), \quad \text{for a.e. $x \in I$} \quad \text{ and } \;\; y(x_1)=y(x_2)=0.
\end{equation}
\item[(iii)]
If $b,c,f \in C(\overline{I})$, then $y$ in (ii) satisfies $y \in C^2(\overline{I})$ and $y$ is a unique classical solution to \eqref{strongequ}, i.e.
\begin{equation*} 
(-\varepsilon y''+ by'+cy)(x) = f(x), \quad \text{ for all $x \in I$}  \quad \text{ and } \;\; y(x_1)=y(x_2)=0.
\end{equation*}
\end{itemize}
\end{theo}
\noindent 
\begin{proof}
(i)
Observe that by \cite[Corollary 8.11]{Br11}, $\rho \in H^{1,1}(I) \cap C(\overline{I})$ and that $\rho(x)>0$ for all $x \in \overline{I}$.
Define a bilinear form $\mathcal{E}$ on $H^{1,2}_0(I) \times H^{1,2}_0(I)$ by 
$$
\mathcal{E}(u,v) := \int_{I} \varepsilon u' v' d\mu  + \int_{I} c uv \,d\mu, \;\;\quad u, v \in H^{1,2}_0(I),
$$
where $\mu = \rho dx$. Then, we obtain from Lemma \ref{embedlem} that for each $u,v \in H^{1,2}_0(I)$
\begin{align*}
|\mathcal{E}(u,v)| &\leq \varepsilon  \left(\int_{I} |u'|^2 \,d\mu \right)^{1/2} \left(\int_{I} |v'|^2\, d\mu \right)^{1/2} + \|\rho c\|_{L^1(I)} \|u\|_{L^{\infty}(I)} \|v\|_{L^{\infty}(I)} \\
&\leq \varepsilon \left(\int_{I} |u'|^2 \,d\mu \right)^{1/2} \left(\int_{I} |v'|^2\, d\mu \right)^{1/2} +|I| \cdot  \|\rho c\|_{L^1(I)} \left( \int_{I} |u'|^2 dx \right)^{1/2}  \left( \int_{I} |v'|^2 dx \right)^{1/2}  \\
& \leq  \underbrace{ (\varepsilon+|I|   \cdot \|c\|_{L^1(I)}) \left(\max_{\overline{I}} \rho \right) }_{=:K}\cdot \left( \int_{I} |u'|^2 dx \right)^{1/2}  \left( \int_{I} |v'|^2 dx \right)^{1/2}  \\
& \leq K \, \|u\|_{H^{1,2}(I)} \|v\|_{H^{1,2}(I)}.
\end{align*}
Moreover,
\begin{align*}
\mathcal{E}(u,u) &\geq \varepsilon \int_{I} |u'|^2 d\mu \geq \varepsilon \left( \min_{\overline{I}} \rho \right) \frac12 \int_{I}|u'|^2 dx+  \varepsilon \left( \min_{\overline{I}} \rho \right) \frac12 |I|^{-1} \|u\|^2_{L^{\infty}(I)}  \\\\
&\geq  \varepsilon \left( \min_{\overline{I}} \rho \right) \frac12 \int_{I}|u'|^2 dx+ \varepsilon \left( \min_{\overline{I}} \rho \right) \frac12 |I|^{-2} \int_{I} |u|^2 dx \\
& \geq \delta \|u\|^2_{H^{1,2}_0(I)},
\end{align*}
where $\displaystyle \delta =\frac{\varepsilon}2 \min_{\overline{I}} \rho \cdot (1 \wedge |I|^{-2})>0$. By the Lax-Milgram theorem (see \cite[Corollary 5.8]{Br11}), there exists a unique $y \in H^{1,2}_0(I)$ such that
$$
\mathcal{E}(y, \psi) = \int_{I} f \psi d\mu, \quad \text{ for all $\psi \in H^{1,2}_0(I)$},
$$
and hence \eqref{variequim} follows. \\
(ii)
Observe that $y' \in H^{1,1}(I) \cap C(\overline{I})$. Indeed, \eqref{variequim} yields that $\rho y' \in H^{1,1}(I) \cap C(\overline{I})$. Since $\frac{1}{\rho} \in H^{1,1}(I) \cap C(\overline{I})$, we have $y' \in H^{1,1}(I) \cap C(\overline{I})$ by the product rule, so that $y \in C^1(\overline{I})$ by Proposition \ref{basicbrezis}(ii).
Moreover, \eqref{variequim} is equivalent to
$$
\int_{I} \left(-\varepsilon y''+ by'+cy\right) \psi \rho dx= \int_{I} \left(-\varepsilon y'' - \frac{\varepsilon \rho'}{\rho}y' + cy\right) \psi \rho dx= \int_{I} f \psi \rho dx, \quad \text{ for all $\psi \in H^{1,2}_0(I)$},
$$
and hence \eqref{strongequdef} holds. Conversely, as the above equivalence if $\overline{y}$ is a strong solution to \eqref{strongequ}, then $\overline{y}$ is a weak solution to \eqref{weaksolu}. Thus, the uniqueness of weak solutions as in (i) implies that 
$\overline{y}=y$ on $I$, and hence the uniqueness of strong solutions to \eqref{strongequ} is shown. \\
(iii) Assume that $b, c, f \in C(\overline{I})$. Then, $\varepsilon y'' = by'+cy-f$ on $I$. Since $by'+cy-f \in C(\overline{I})$, we have $y'' \in C(\overline{I})$ by Proposition \ref{basicbrezis}(ii) . Therefore, $y \in C^2(\overline{I})$ is a unique classical solution to \eqref{strongequ}.
\end{proof}

\begin{theo}[Energy estimates] \label{energyestim} 
Let $\varepsilon>0$ be a constant, $b \in L^1(I)$, $c \in L^1(I)$ with $c \geq 0$ on $I$ and $f \in L^1(I)$, where $I$ is a bounded open interval.
Let $y \in H^{1,2}_0(I) \cap H^{2,1}(I) \cap C^1(\overline{I})$ be a unique strong solution to \eqref{strongequdef} as in Theorem \ref{basictheo}(ii). Then, it holds that
\begin{equation} \label{firstestim}
\|y'\|_{L^2(I)} \leq \left( \frac{\max_{\overline{I}} \rho}{ \varepsilon \min_{\overline{I}} \rho}  \right) |I|^{1/2} \|f\|_{L^1(I)},
\end{equation}
where $\rho$ is the function defined in \eqref{defnrhofun}. In particular, if $f \in L^2(I)$, then
\begin{align}
\|y\|_{L^2(I)}  &\leq  |I|^{1/2}\|y\|_{L^{\infty}(I)} \leq |I| \|y'\|_{L^{2}(I)} \leq
\left( \frac{\max_{\overline{I}} \rho}{\varepsilon \min_{\overline{I}} \rho}  \right)  |I|^{3/2} \|f\|_{L^1(I)} \label{linfestim} 	\\
&\leq \left( \frac{\max_{\overline{I}} \rho}{\varepsilon \min_{\overline{I}} \rho}  \right)  |I|^{2} \|f\|_{L^2(I)}.  \nonumber 
\end{align}
\end{theo}
\noindent
\begin{proof}
Substituting $y$ for $\psi$ in \eqref{variequim} and using Lemma \ref{embedlem} and the H\"{o}lder inequality, we get
\begin{align*} \label{variequim2}
\varepsilon \left(\min_{\overline{I}} \rho\right) \int_{I} |y'|^2 dx &\leq \varepsilon \int_{I} |y'|^2 \rho dx \leq \varepsilon \int_{I} |y'|^2 \rho dx + \int_I c y^2 \rho dx = \int_{I} f y \rho dx \\
&\leq \left(\max_{\overline{I}} \rho \right) \|y\|_{L^{\infty}(I)} \|f \|_{L^1(I)} \leq \left( \max_{\overline{I}} \rho \right)|I|^{1/2} \left( \int_{I} |y'|^2 dx \right)^{1/2} \|f\|_{L^{1}(I)},
\end{align*}
and hence \eqref{firstestim} follows. The rest follows by Lemma \ref{embedlem} and the H\"{o}lder inequality.
\end{proof}

\begin{theo}  \label{exunisol}
Let $\varepsilon>0$ be a constant, $b \in L^1(I)$, $c \in L^1(I)$ with $c \geq 0$ on $I$ and $\widetilde{f} \in L^1(I)$, where $I=(x_1, x_2)$ is a bounded open interval. Let $p,q \in \mathbb{R}$ be given.
Then, there exists a unique strong solution $\widetilde{y} \in H^{2,1}(I) \cap C^1(\overline{I})$ to the following boundary value problem
\begin{equation} \label{bvpstrong}
\left\{
\begin{alignedat}{2}
& -\varepsilon\widetilde{y}''+b \widetilde{y}' +c \widetilde{y}=  \widetilde{f} \quad \text{ on }  I, \\
&\widetilde{y}(x_1)=p, \quad \widetilde{y}(x_2)=q,
\end{alignedat} \right.
\end{equation}
i.e.
\begin{equation*} \label{strongequdef2}
(-\varepsilon \widetilde{y}''+ b\widetilde{y}'+c\widetilde{y})(x) = \widetilde{f}(x), \quad \text{for a.e. $x \in I$} \quad \text{ and } \;\; \widetilde{y}(x_1)=p,\;\; \widetilde{y}(x_2)=q.
\end{equation*}
In particular, if $\widetilde{f} \in L^2(I)$, then $\widetilde{y}$ fulfills the following estimate:
\begin{align*}
 \|\widetilde{y}\|_{L^{2}(I)} \leq  K_1 \|\widetilde{f}\|_{L^2(I)} +K_2 (|p| \vee |q|),
\end{align*}
where $K_1=\left( \frac{\max_{\overline{I}} \rho}{\varepsilon\min_{\overline{I}} \rho}  \right) |I|^2$ and $K_2=\left( \frac{\max_{\overline{I}} \rho}{ \varepsilon \min_{\overline{I}} \rho}  \right) \left( |I|^{1/2} \|b\|_{L^1(I)}  +  |I|^{3/2} \|c\|_{L^1(I)}\right) + |I|^{1/2}$ and $\rho$ is the function defined in \eqref{defnrhofun}.
Moreover, if $b, c \in C(\overline{I})$, then $\widetilde{y} \in C^2(\overline{I})$ and it is a unique classical solution to \eqref{bvpstrong}.
\end{theo}
\noindent
\begin{proof}
Let $\ell(x):=\frac{q-p}{x_2-x_1}(x-x_1)+p$, \,$x \in \mathbb{R}$. Then, it is easy to check that
$\|\ell\|_{L^{\infty}(I)} \leq |p|\vee|q|$ and $\|\ell'\|_{L^{\infty}(I)} = |I|^{-1} |p-q|$.
Then, $\widetilde{y} \in H^{2,1}(I) \cap C(\overline{I})$ is a strong solution to \eqref{bvpstrong} if and only if $y \in H^{1,2}_0(I) \cap H^{2,1}(I) \cap C(\overline{I})$ is a strong solution to 
\begin{equation} \label{strongequ2}
\left\{
\begin{alignedat}{2}
& -\varepsilon y''+b y' +c y= \widetilde{f}- b \ell' -c\ell \quad \text{ on }  \,I, \\
&y(x_1)=0, \quad y(x_2)=0,
\end{alignedat} \right.
\end{equation}
where $y = \widetilde{y}-\ell$. Indeed, by Theorem \ref{basictheo} there exists a unique strong solution $y\in H^{1,2}_0(I) \cap H^{2,1}(I) \cap C^1(\overline{I})$ to \eqref{strongequ2}. Define $\widetilde{y}:=y+\ell$.
Then, $\widetilde{y}$ is a unique strong solution $\widetilde{y} \in H^{2,1}(I) \cap C^1(\overline{I})$ to \eqref{bvpstrong}. Moreover, \eqref{linfestim} in Theorem \ref{energyestim} implies that
$$
|I|^{1/2}\|\widetilde{y}-\ell\|_{L^{\infty}(I)} \leq 
\left( \frac{\max_{\overline{I}} \rho}{\varepsilon \min_{\overline{I}} \rho}  \right)  |I|^{3/2} \|\widetilde{f}- b \ell' -c\ell\|_{L^1(I)}.
$$
In particular, if $\widetilde{f} \in L^2(I)$, then
\begin{align*}
&\|\widetilde{y}\|_{L^2(I)} \leq |I|^{1/2} \| \widetilde{y}\|_{L^{\infty}(I)} \leq |I|^{1/2} \| \widetilde{y}-\ell \|_{L^{\infty}(I)}+ |I|^{1/2} \| \ell \|_{L^{\infty}(I)} \\
&\leq\left( \frac{\max_{\overline{I}} \rho}{\varepsilon \min_{\overline{I}} \rho}  \right) \left( |I|^2  \|\widetilde{f}\|_{L^2(I)}  + |I|^{1/2}\|b\|_{L^1(I)} |q-p|  + |I|^{3/2}\|c\|_{L^1(I)} (|p| \vee |q|)	\right)  + |I|^{1/2} (|p| \vee |q|) \\
&\leq  \left( \frac{\max_{\overline{I}} \rho}{\varepsilon \min_{\overline{I}} \rho}  \right) |I|^2 \cdot \|\widetilde{f}\|_{L^2(I)} +\left(\Big( \frac{\max_{\overline{I}} \rho}{\varepsilon \min_{\overline{I}} \rho}  \Big) \Big( |I|^{1/2} \|b\|_{L^1(I)}  +  |I|^{3/2} \|c\|_{L^1(I)}\Big) + |I|^{1/2} \right) (|p| \vee |q|),
\end{align*}
as desired.
\end{proof}

\begin{theo}  \label{contierres}
Let $\varepsilon>0$ be a constant, $b \in L^2(I)$, $c \in L^2(I)$ with $c \geq 0$ on $I$ and $\widetilde{f} \in L^2(I)$, where $I=(x_1, x_2)$ is a bounded open interval. Let $p,q \in \mathbb{R}$ be given and $\widetilde{y} \in H^{2,1}(I) \cap C^1(\overline{I})$ be the unique strong solution to
\eqref{bvpstrong} as in Theorem \ref{exunisol}. 
Let $\Phi \in H^{2,2}(I) \cap C^1(\overline{I})$. Then,
\begin{align*}
\| \widetilde{y}-\Phi \|_{L^2(I)} \leq K_1 \big\| \widetilde{f} - \mathcal{L}[\Phi]   \big \|_{L^2(I)}+ K_2\bigg( \big|p-\Phi(x_1)\big|  \vee \big|q-\Phi(x_2)\big| \bigg),
\end{align*}
where $K_1, K_2>0$ are constants in Theorem \ref{exunisol} and $\mathcal{L}$ is a differential operator defined by
\begin{equation} \label{mathcalloper}
\mathcal{L} [w]:= -\varepsilon w''+bw'+cw, \quad w \in H^{2,1}(I).
\end{equation}
\end{theo}
\noindent
\begin{proof}
Let $u := \widetilde{y}-\Phi$. 
Then, $u \in H^{2,2}(I) \cap C^1(\overline{I})$. Since $\widetilde{f}-\mathcal{L}[\Phi] \in L^2(I)$, it follows from Theorem \ref{basictheo} that $u$ is a (unique) strong solution to the following problem:
\begin{equation*} \label{bvpstrong3}
\left\{
\begin{alignedat}{2}
& -\varepsilon u''+b u' +c u= \widetilde{f}-\mathcal{L}[\Phi] \quad \text{ on }  I, \\
&\; u(x_1)=p-\Phi(x_1), \quad u(x_2)=q-\Phi(x_2).
\end{alignedat} \right.
\end{equation*}
By Theorem \ref{exunisol},
$$
\|u\|_{L^2(I)} \leq K_1 \| \widetilde{f}-\mathcal{L}[\Phi] \|_{L^2(I)} + K_2\bigg( \big|p-\Phi(x_1)\big|  \vee \big|q-\Phi(x_2)\big| \bigg),
$$
as desired.
\end{proof}

\begin{theo} \label{pinn2eneres}
Let $\varepsilon>0$ be a constant, $b \in L^2(I)$, $c \in L^2(I)$ with $c \geq 0$ on $I$ and $\widetilde{f} \in L^2(I)$, where $I=(x_1, x_2)$ is a bounded open interval. Let $p,q \in \mathbb{R}$ be given.
Let $\widetilde{y} \in H^{2,1}(I) \cap C^1(\overline{I})$ be the unique strong solution to
\eqref{bvpstrong} as in Theorem \ref{exunisol}. Let $\Phi_A \in H^{1,2}_0(I) \cap H^{2,2}(I) \cap C^1(\overline{I})$ and $\ell(x):=\frac{q-p}{x_2-x_1}(x-x_1)+p$, $x \in \mathbb{R}$. Then,
\begin{align*}
\big\| \widetilde{y} - (\Phi_A +\ell)  \big\|_{L^2(I)} \leq \left( \frac{\max_{\overline{I}} \rho}{\varepsilon\min_{\overline{I}} \rho}  \right) |I|^2 \big\|\widetilde{f} -\mathcal{L}[\Phi_A+\ell]\big\|_{L^2(I)},
\end{align*}
where $\mathcal{L}$ is a differential operator defined by \eqref{mathcalloper}.
\end{theo}

\noindent
\begin{proof}
Since $(\Phi_A +\ell)(x_1)=p$ and $(\Phi_A +\ell)(x_2)=q$, the assertion follows from Theorem \ref{contierres} where 
$\Phi$ there is replaced by $\Phi_A+\ell$ here.
\end{proof}

\centerline{}
\subsection{Numerical experiments for PINN with energy estimates} \label{basicexperime}
\noindent
Let $\varepsilon>0$ be a constant, $b \in C(\overline{I})$, $c \in C(\overline{I})$ with $c \geq 0$ on $I$ and $\widetilde{f} \in C(\overline{I})$, where $I=(x_1, x_2)$ is a bounded open interval. Let $p,q \in \mathbb{R}$ be given and $\rho$ be the function defined in \eqref{defnrhofun}.
Let $\widetilde{y} \in C^2(\overline{I})$ be the unique classical solution to
\eqref{bvpstrong} as in Theorem \ref{exunisol}.  Let $(X_i)_{i \geq 1}$ be a sequence of independent and identically distributed random variables on a probability space $(\Omega, \mathcal{F}, \mathbb{P})$  that has a continuous uniform distribution on $I$.
Let $\Phi \in C^2(\overline{I})$. Then, by using Theorem \ref{contierres} with the Monte Carlo integration as in \eqref{montecarint}, we obtain that for sufficiently large $n \in \mathbb{N}$
\begin{align}
&{\text{\sf (PINN I):}}  \nonumber \\
&  \underbrace{\frac{1}{n} \sum_{i=1}^n |\widetilde{y}(X_i)-\Phi(X_i)|^2}_{=: {\sf Error}} \leq \left(\frac{\max_{\overline{I}} \rho}{ \varepsilon \min_{\overline{I}} \rho} \right)^2  
 |I|^4  \cdot \left(  \frac{1}{n} \sum_{i=1}^n \Big(\widetilde{f}(X_i)-\mathcal{L}[\Phi](X_i) \Big)^2 \right)  \nonumber \\
& \quad + \left(\Big(\frac{\max_{\overline{I}} \rho}{\varepsilon \min_{\overline{I}} \rho}
 \Big) \Big(\|b\|_{L^1(I)}+ |I| \|c\|_{L^1(I)}\Big)+1 \right)^2 \bigg( \underbrace{ \big|p-\Phi(x_1)\big|^2  + \big|q-\Phi(x_2)\big|^2}_{=:{\sf Boundary \; Loss}} \bigg)  \nonumber \\
& \leq  \Bigg(\left(\frac{\max_{\overline{I}} \rho}{ \varepsilon \min_{\overline{I}} \rho} \right)^2  
 |I|^4+ \left(\Big(\frac{\max_{\overline{I}} \rho}{\varepsilon \min_{\overline{I}} \rho}
 \Big) \Big(\|b\|_{L^1(I)}+ |I| \|c\|_{L^1(I)}\Big)+1 \right)^2 \Bigg) \times \nonumber \\
& \qquad \quad \Bigg(  \underbrace{\frac{1}{n} \sum_{i=1}^n \Big(\widetilde{f}(X_i)-\mathcal{L}[\Phi](X_i) \Big)^2 +\, \big|p-\Phi(x_1)\big|^2  + \big|q-\Phi(x_2)\big|^2}_{=: \sf Loss}  \Bigg) , \qquad \text{very likely.} \quad \label{pinn1erbdd}
\end{align}

\centerline{}
\noindent
On the other hand, let $\Phi_A \in H^{1,2}_0(I) \cap C^2(\overline{I})$ and $\ell(x):=\frac{q-p}{x_2-x_1}(x-x_1)+p$, $x \in \mathbb{R}$. 
Then, by using Theorem \ref{pinn2eneres} with the Monte Carlo integration as in \eqref{montecarint}, we find that for sufficiently large $n \in \mathbb{N}$
\begin{align}
&{\text{\sf (PINN II):}} \nonumber \\
&  \underbrace{\frac{1}{n} \sum_{i=1}^n |\widetilde{y}(X_i)-(\Phi_A+\ell)(X_i)|^2}_{=: {\sf Error}} \leq \left(\frac{\max_{\overline{I}} \rho}{\varepsilon \min_{\overline{I}} \rho}
 \right)^2  |I|^4 \cdot  \underbrace{\frac{1}{n} \sum_{i=1}^n \Big(\widetilde{f}(X_i)-\mathcal{L}[\Phi_A+\ell](X_i) \Big)^2}_{=: {\sf Loss}}, \quad \text{ very likely}. \qquad \qquad \label{engpinn2erbdd}
\end{align} 
\noindent
Let $N(x; \theta)$ be a real-valued neural network defined on $\mathbb{R}$ and let $\chi(x)=-(x-x_1)(x-x_2)$.
 In the following example we will apply {\sf (PINN I)} and {\sf (PINN II)} where $\Phi(x)$ and  $\Phi_A(x)+\ell(x)$ are replaced by
 $N(x;\theta)$ and $\chi(x) N(x; \theta)+\ell(x)$, respectively. Both methods can reduce the error by reducing {\sf Loss}, but {\sf (PINN I)} has a limitation in that the discrepancy of the boundary value remains. In addition, the following example confirms that {\sf (PINN II)} is superior to {\sf (PINN I)} in terms of error reduction performance.
 \newpage
\begin{exam} \label{pinn12compare}
\rm
Let $I=(0, 1)$. Let $\varepsilon:=1$, $c(x)=10$, $b(x)=x$ and $\widetilde{y}(x)=x(1-x) e^x$,  $x \in \mathbb{R}$. 
Note that
$$
\widetilde{y}'(x) = (1-x-x^2)e^x, \; \;\;  \widetilde{y}''(x)=(-3x -x^2)e^x, \quad x \in \mathbb{R}.
$$
Let $\widetilde{f}(x)= - \widetilde{y}''(x) + b(x) \widetilde{y}'(x) + c(x) \widetilde{y}(x)= \big( (4+c)x -cx^2-x^3   \big) e^x$, \,$x \in \mathbb{R}$. Then, $\widetilde{y}$ is a unique strong solution to \eqref{bvpstrong} with $p=q=0$. \\ \\
For the following experiment, a three-layer neural network with an architecture configured as $(1, 16, 32, 1)$ was utilized. The hyperbolic tangent function ($\tanh$) was used as the activation function throughout the network. The Adam optimizer was used with an initial learning rate of 0.1 and batch size of 10000. To optimize training, an exponential learning rate scheduler was attached, which methodically decreased the learning rate at each epoch to ensure convergence.
\begin{figure}[!h]      
     \centering
     \begin{subfigure}[b]{0.42\textwidth} 
         \centering
         \includegraphics[width=\textwidth]{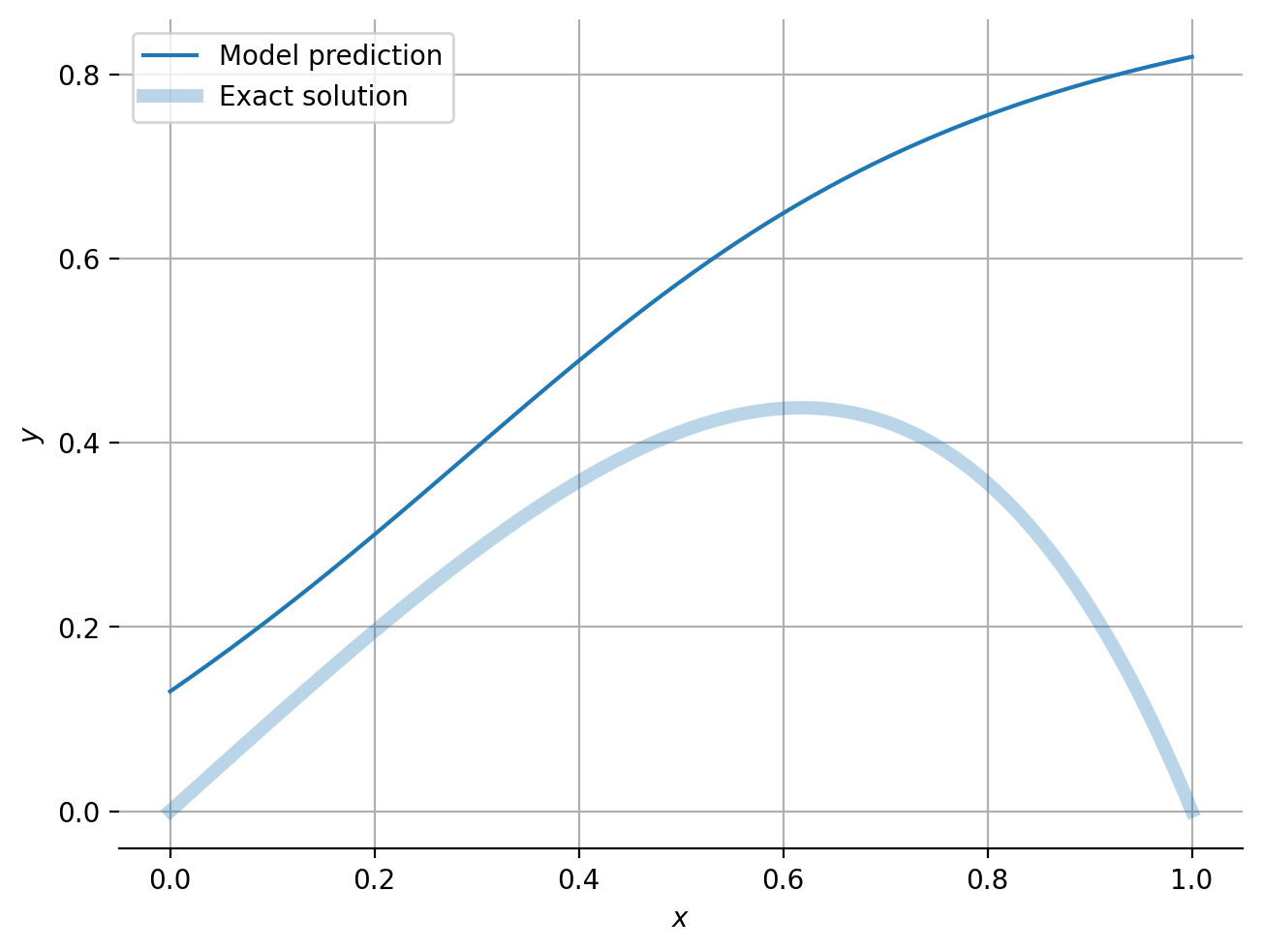}
	\caption{Model prediction comparison with exact solution after 500 epochs  {\sf (PINN I)}}
	 \label{figpinn1per(a)}
     \end{subfigure}
     \hfill
     \begin{subfigure}[b]{0.42\textwidth} \label{figpinn1per(b)}
         \centering
         \includegraphics[width=\textwidth]{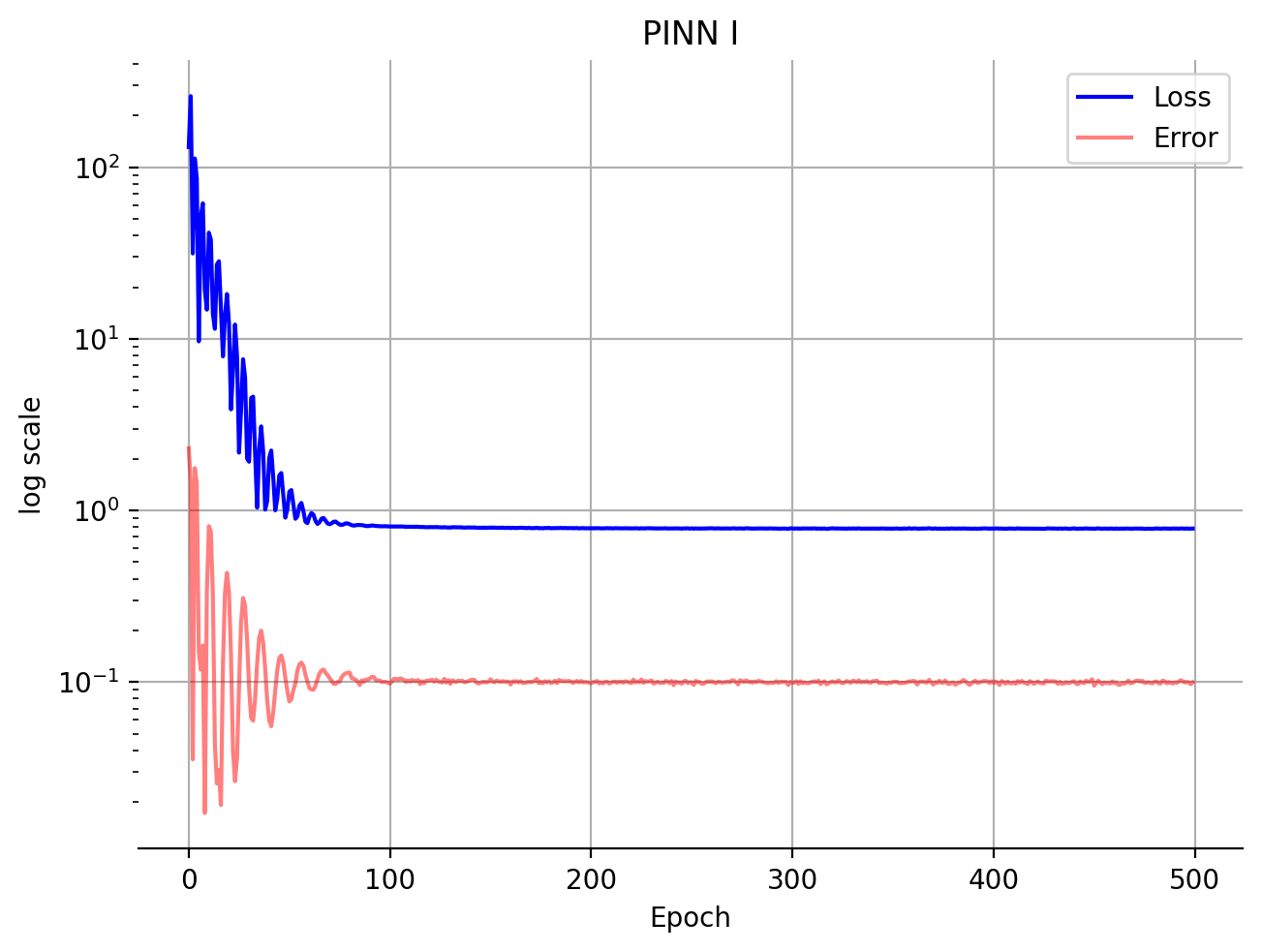}
	\caption{The behaviors of {\sf Loss} and {\sf Error} as epoch increases {\sf (PINN I)}}
         \label{fig:result}
     \end{subfigure}
     \caption{Not effective {\sf (PINN I)}}
     \label{figpinn1per}
\end{figure}      
\begin{figure}[!h]         
     \centering
     \begin{subfigure}[b]{0.42\textwidth}
         \centering
         \includegraphics[width=\textwidth]{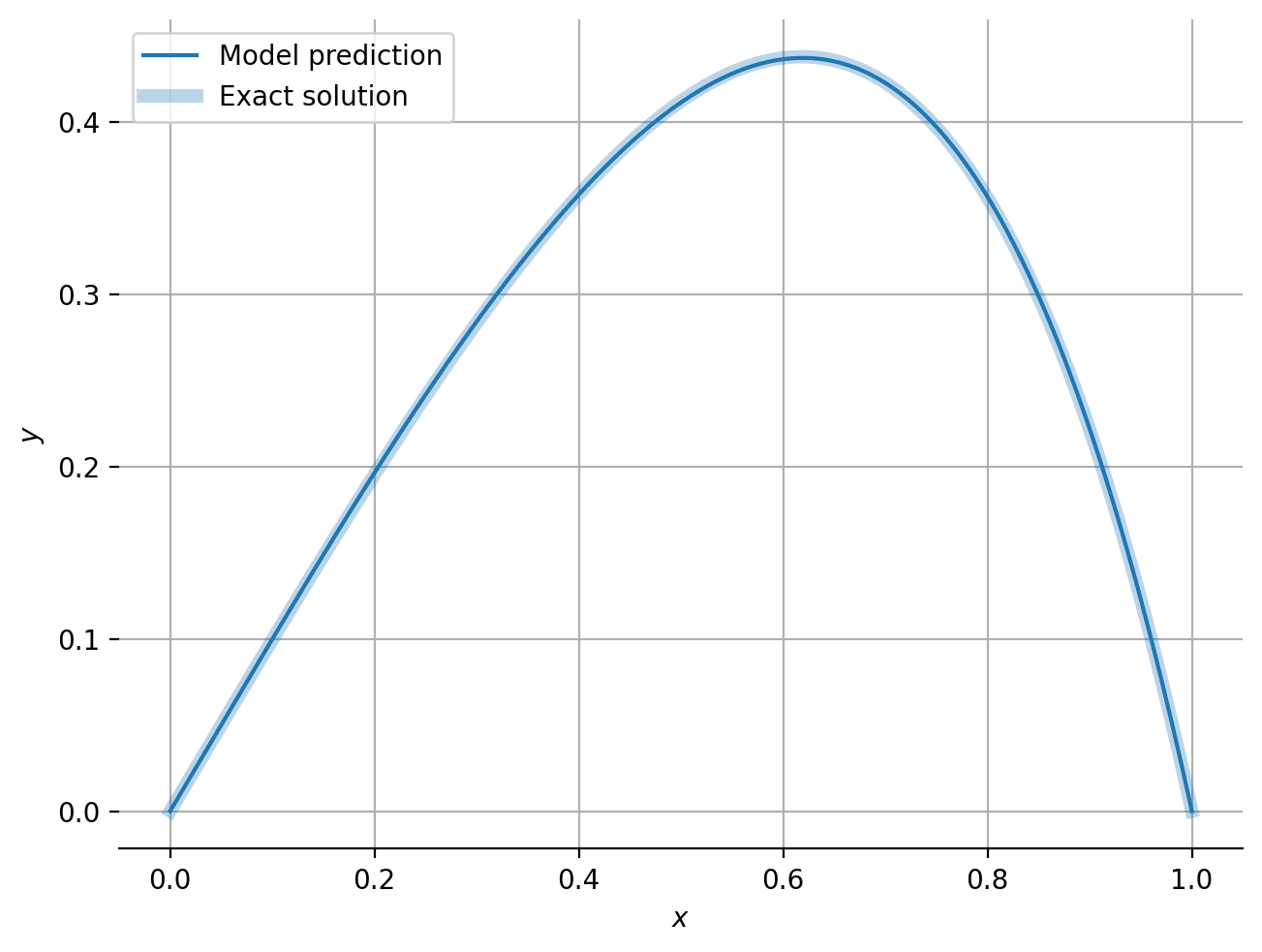}
	\caption{Model prediction comparison with exact solution after 500 epochs {\sf (PINN II)}}
     \end{subfigure}
     \hfill
     \begin{subfigure}[b]{0.42\textwidth}
         \centering
         \includegraphics[width=\textwidth]{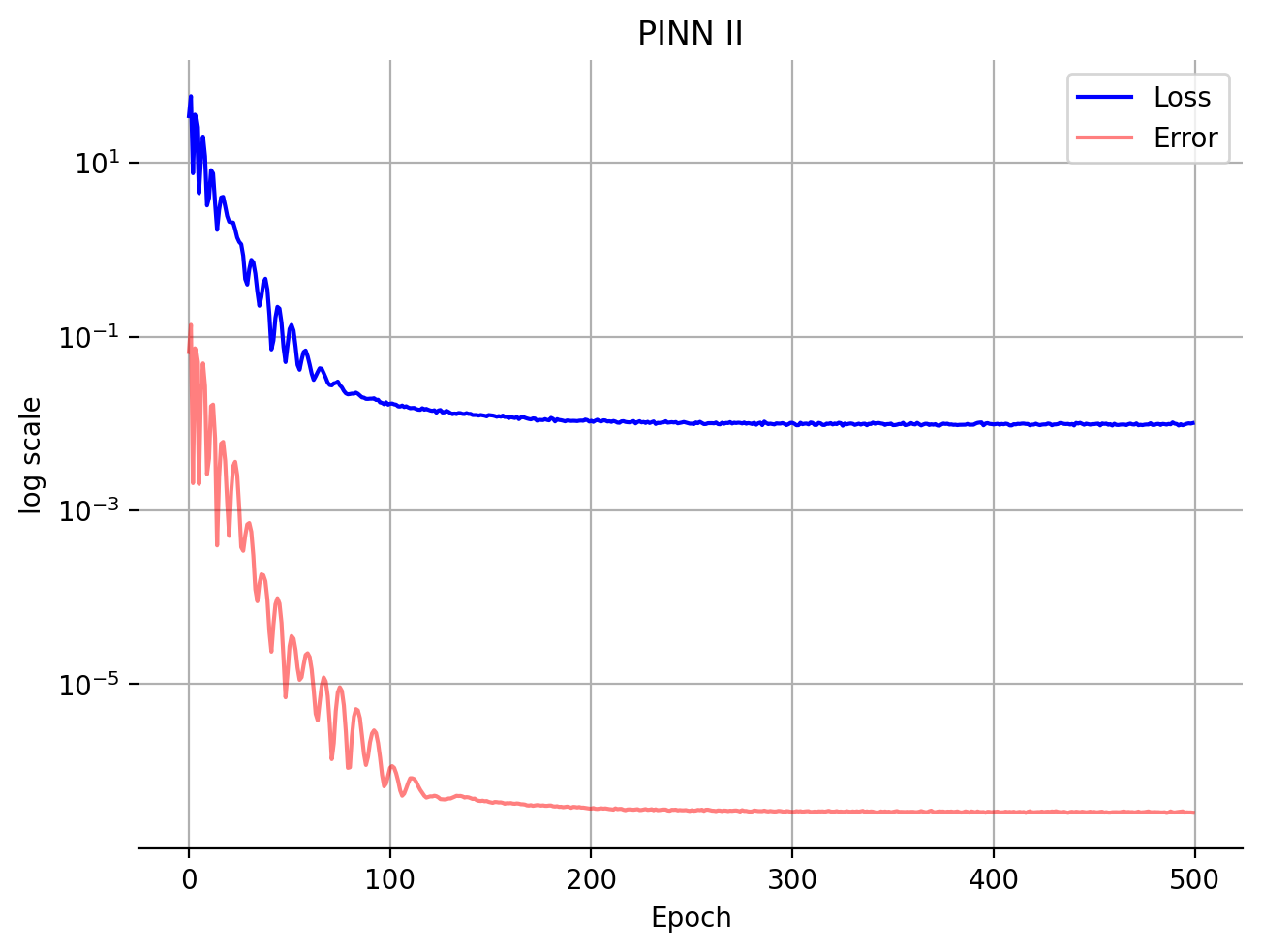}
	\caption{The behaviors of {\sf Loss} and {\sf Error} as epoch increases {\sf (PINN II)}}
         \label{fig:result}
     \end{subfigure}
     \caption{Successful {\sf (PINN II)}}
	\label{figpinn2per}
\end{figure}
\begin{figure}[!h]
     \centering
     \includegraphics[width=0.43\textwidth]{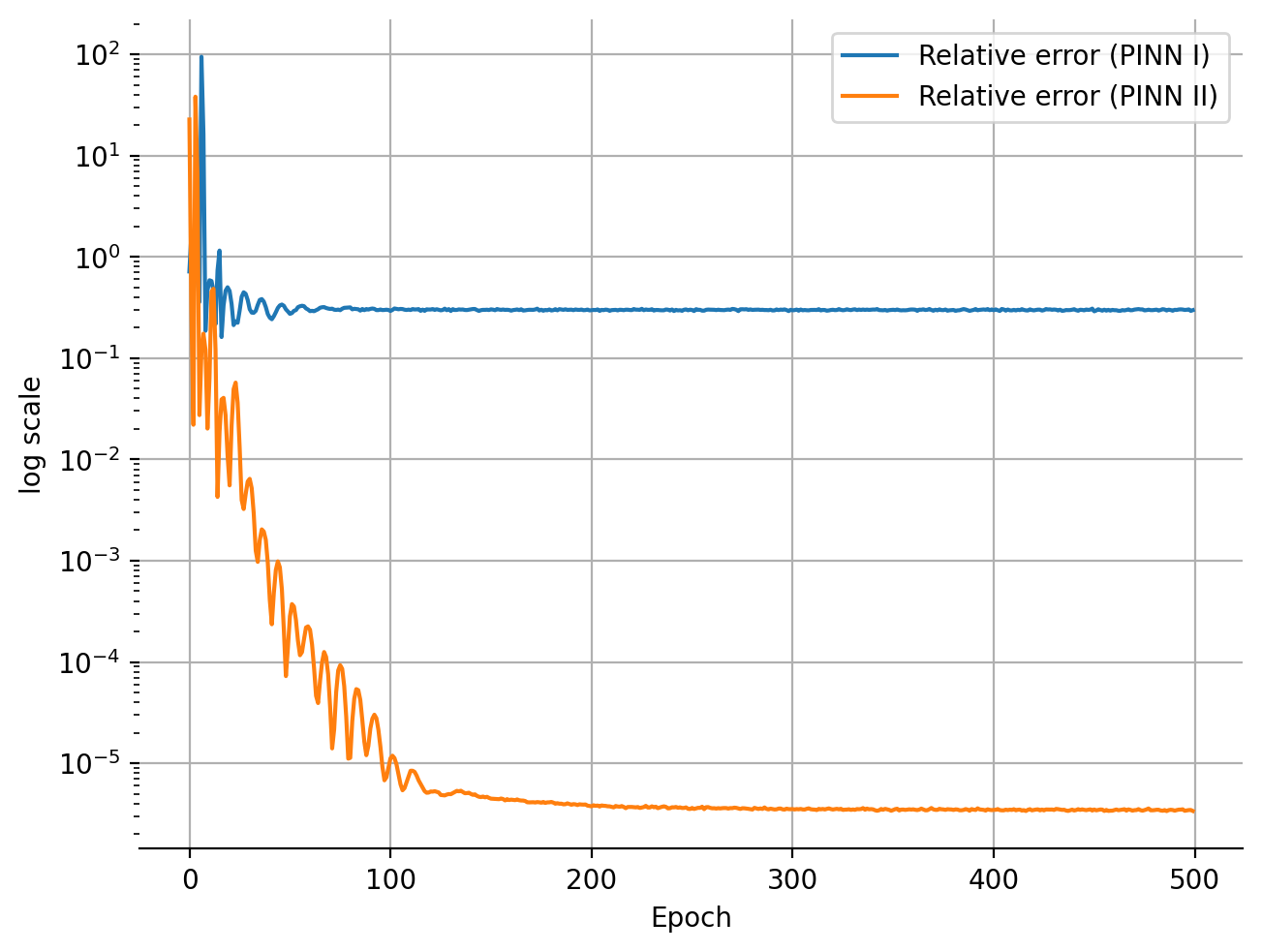}
	\caption{Comparison of {\sf Relative error} between two methods as epoch increases}
     \label{fig:Comp}
\end{figure}
\text{}\\
Figure \ref{figpinn1per} (a) describes the graph of the model prediction (neural network $N(x; \theta))$ and the exact solution ($\widetilde{y}(x)$) on $I$ where the model prediction does not accurately capture the exact solution. Whereas, Figure \ref{figpinn2per} (a) describes the graph of the model prediction ($x(1-x)N(x, \theta))$ and the exact solution ($\widetilde{y}(x)$) for which the model prediction accurately captures the exact solution. Figure \ref{figpinn1per} (b) and Figure \ref{figpinn2per} (b)  describe the graphs of {\sf Loss} and {\sf Error} along epoch increases, respectively and we can observe that {\sf (PINN II)} is more efficient than {\sf (PINN I)} in reducing {\sf Error} and {\sf Loss}.  Furthermore, Figure \ref{fig:Comp} describes that {\sf (PINN II)} gives a more precise model prediction compared to {\sf (PINN I)} in the perspective of {\sf Relative error}. The {\sf Relative errors} are defined as
\begin{align*}
{\sf Relative\,error \,(PINN\, I)}&=\frac{\sum_{i=1}^n |\widetilde{y}(X_i)-\Phi(X_i)|^2}{\sum_{i=1}^n |\widetilde{y}(X_i)|^2}, \\
{\sf Relative\,error \,(PINN\, II)}&=\frac{\sum_{i=1}^n |\widetilde{y}(X_i)-(\Phi_A+\ell)(X_i)|^2}{\sum_{i=1}^n |\widetilde{y}(X_i)|^2}.
\end{align*}
\end{exam}
\noindent
In Example \ref{pinn12compare}, we can observe that {\sf (PINN I)} exhibit significant boundary discrepancies, where {\sf Boundary Loss} is larger than {\sf Error}. In contrast, {\sf (PINN II)} accurately captures the exact solution. Therefore, from now on we exclusively use {\sf (PINN II)} where the trial function $\Phi_A+\ell $ exactly satisfies the boundary conditions of \eqref{bvpstrong}.

\centerline{}

\begin{exam} \label{pinn12newex}
\rm
Let $I=(0, 1)$. Let $\varepsilon:=1$, $c(x)=\lambda$, $b(x)=\tan \frac{\pi x}{6}$ and $\widetilde{y}(x)=\sin \frac{\pi x}{6}$,  $x \in \mathbb{R}$. 
Note that
$$
\widetilde{y}\,'(x) = \frac{\pi}{6}\cos \frac{\pi x}{6}, \; \; \;\;  \widetilde{y}\,''(x)= -\frac{\pi^2}{36} \sin \frac{\pi}{6}, \quad x \in \mathbb{R}.
$$
Let $\widetilde{f}(x)= - \widetilde{y}''(x) + b(x) \widetilde{y}'(x) + c(x) \widetilde{y}(x)= \Big(\frac{\pi^2}{36} + \frac{\pi}{6} +\lambda\Big) \sin \frac{\pi}{6}x$, \,\,$x \in \mathbb{R}$. Then, $\widetilde{y}$ is a unique strong solution to \eqref{bvpstrong} with $p=0$ and $q=\frac{1}{2}$.
\end{exam}
\noindent
For the following experiment, a three-layer neural network with an architecture configured as $(1, 50, 50, 1)$ was utilized to describe a solution for a more complex differential equation. All experimental details are identical to those described in Example \ref{pinn12compare}.\\

\begin{figure}[!h]
     \centering
     \begin{subfigure}[b]{0.43\textwidth}
         \centering
         \includegraphics[width=\textwidth]{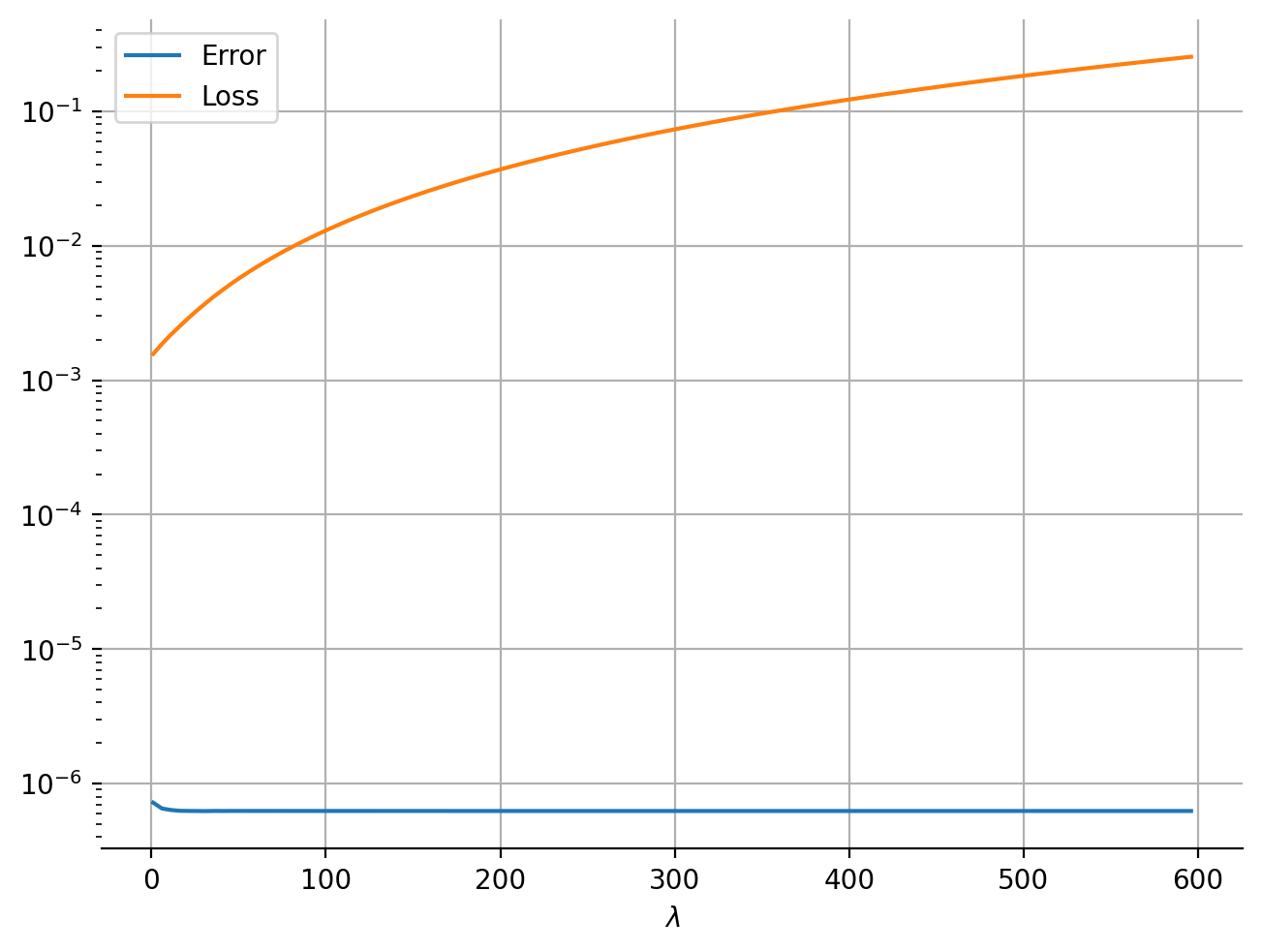}
	\caption{{\sf Loss} and {\sf Error} as $\lambda$ increases after 500 epochs as $\lambda$ increases}
         \label{fig:train}
     \end{subfigure}
     \hfill
     \begin{subfigure}[b]{0.43\textwidth}
         \centering
         \includegraphics[width=\textwidth]{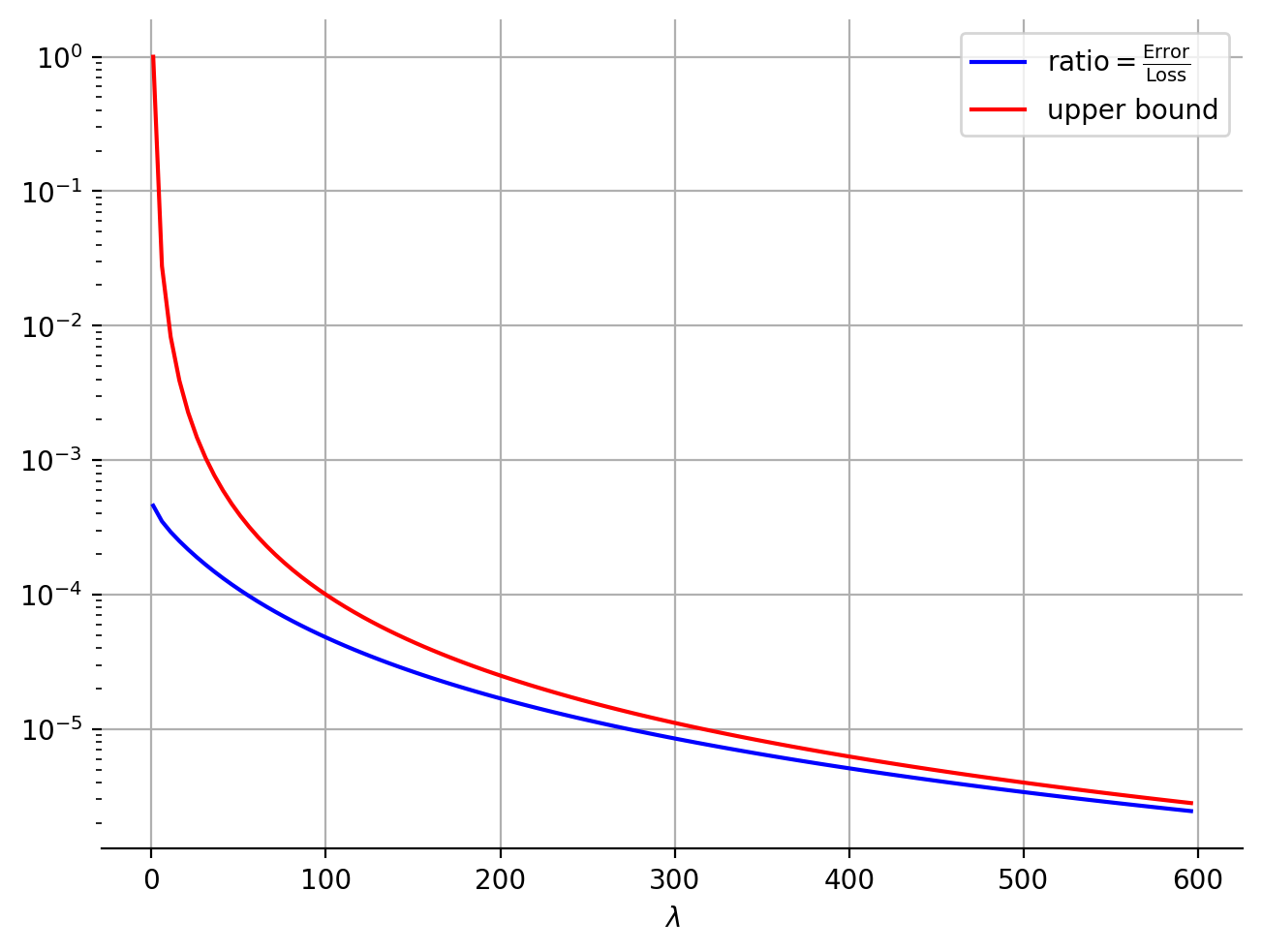}
	\caption{$\frac{\sf Error}{\sf Loss}$ bounded by the upper bound $\frac{1}{\lambda^2}$ as $\lambda$ increases}
         \label{fig:result}
     \end{subfigure}
     \caption{As $\lambda$ increases, {\sf Loss} increases but {\sf Error} remains robust due to the decrease of $\frac{\sf Error}{\sf Loss}$}
     \label{figweighestim}
\end{figure}
\noindent
Increasing the zero-order coefficient $\lambda$ leads to a more robust {\sf Error}, although the {\sf Loss} increases as well. We investigated this phenomenon and found that the ratio $\frac{\sf Error}{\sf Loss}$ decreases, and this ratio is sharply bounded by the function $\frac{1}{\lambda^2}$. However, this observation is not explained by our energy estimates \ref{engpinn2erbdd}, as it does not account for zero-order terms. This discrepancy suggests the need for further investigation into this topic.
\section{Error analysis via $L^2(I, \mu)$-contraction estimates} \label{mucontraestim}
\subsection{$L^2(I, \mu)$-contraction estimates}
\begin{theo}[$L^2(I, \mu)$-contraction estimates] \label{contraesthm}
 Let $\varepsilon>0$ be a constant, $b \in L^2(I)$, $c \in L^2(I)$ with $c \geq \lambda$ on $I$ for some constant $\lambda>0$ and $f \in L^2(I)$, where $I$ is a bounded open interval. Let
$\rho$ be a function defined in \eqref{defnrhofun},
where $b$ above is considered as the zero extension of $b$ in $\mathbb{R}$. Let $y \in H^{1,2}_0(I) \cap H^{2,1}(I) \cap C^1(\overline{I})$ be a unique strong solution to \eqref{strongequ} as in Theorem \ref{basictheo}. Then,
\begin{equation} \label{contraresoles}
\| y\|_{L^2(I, \mu)} \leq \frac{1}{\lambda} \|f\|_{L^2(I, \mu)},
\end{equation}
where $\mu=\rho dx$. In particular,
\begin{equation*} \label{contrardxchan}
\| y\|_{L^2(I)} \leq \left(\frac{\max_{\overline{I}}\rho}{\min_{\overline{I}} \rho } \right)^{1/2}\frac{1}{\lambda} \|f\|_{L^2(I, \mu)} \leq \exp\left(\int_{I} \frac{1}{2\varepsilon} |b| dx \right) \frac{1}{\lambda} \|f\|_{L^2(I)}.
\end{equation*}
\end{theo}
\noindent
\begin{proof}
By Theorem \ref{basictheo}, the unique strong solution $y$ to \eqref{strongequ} satisfies that
\begin{equation*} \label{variequim3}
\int_{I} \varepsilon y' \psi' \rho dx + \int_I c y \psi \rho dx  = \int_{I} f \psi \rho dx, \quad \text{ for all $\psi \in H^{1,2}_0(I)$}.
\end{equation*}
Substituting $y$ for $\psi$, we get
$$
\lambda \int_{I} y^2 \rho dx \leq \int_{I} \varepsilon |y'|^2 \rho dx + \int_{I} cy^2 \rho dx = \int_{I} f y \rho dx \leq \left(\int_{I} f^2 \rho dx \right)^{1/2} \left( \int_{I} y^2 \rho dx \right)^{1/2},
$$
and hence the result follows.
\end{proof}

\begin{theo} \label{contrastabiles}
 Let $p, q \in \mathbb{R}$,  $\varepsilon>0$ be constants, $b \in L^2(I)$, $c \in L^2(I)$ with $c \geq \lambda$ on $I=(x_1, x_2)$ for some constant $\lambda>0$ and $f \in L^2(I)$, where $I$ is a bounded open interval. Let
$\rho$ be a function defined in \eqref{defnrhofun}.
Then, the  unique strong solution $\widetilde{y} \in H^{2,1}(I) \cap C^1(\overline{I})$ to \eqref{bvpstrong} as in 
Theorem \ref{exunisol} satisfies
$$
\|\widetilde{y}\|_{L^2(I, \mu)} \leq \frac{1}{\lambda} \| \widetilde{f} \|_{L^2(I, \mu)} +K_3 (|p| \vee |q|),
$$
where $\mu=\rho dx$ and  $K_3= \frac{1}{\lambda}\|b\|_{L^2(I, \mu)}  |I|^{-1}+ \frac{1}{\lambda} \|c\|_{L^2(I, \mu)}+|I|^{1/2}  (\max_{\overline{I}} \rho)^{1/2}$. In particular,
$$
\|\widetilde{y}\|_{L^2(I)} \leq \frac{1}{\lambda}  \left(\frac{\max_{\overline{I}}\rho}{\min_{\overline{I}} \rho } \right)^{1/2}\| \widetilde{f} \|_{L^2(I)} +\left(\frac{\max_{\overline{I}}\rho}{\min_{\overline{I}} \rho } \right)^{1/2} \widetilde{K}_3 (|p| \vee |q|),
$$
where $\widetilde{K}_3=\frac{1}{\lambda}\|b\|_{L^2(I)}  |I|^{-1}+ \frac{1}{\lambda} \|c\|_{L^2(I)}+|I|^{1/2}$.
\end{theo}

\noindent
\begin{proof}
Let $\ell(x):=\frac{q-p}{x_2-x_1}(x-x_1)+p$, $x \in \mathbb{R}$. Then, we can check that
$\|\ell\|_{L^{\infty}(I)} \leq |p|\vee|q|$ and $\|\ell'\|_{L^{\infty}(I)} \leq |I|^{-1} |p-q|$. Let $y:=\widetilde{y}-\ell$. Then, 
$y \in H^{1,2}_0(I) \cap H^{2,1}(I) \cap C^1(\overline{I})$ is a strong solution to \eqref{strongequ2} where $\widetilde{f}$ is replaced by $\widetilde{f}- b \ell' -c\ell$. Then, we derive from \eqref{contraresoles} in Theorem \ref{contraesthm} that $y = \widetilde{y}-\ell$ and that
\begin{align*}
\| \widetilde{y}-\ell \|_{L^2(I, \mu)}&= \|y\|_{L^2(I, \mu)} \leq \frac{1}{\lambda} \big\|\widetilde{f}- b \ell' -c\ell \big\|_{L^2(I, \mu)}\\
&\leq \frac{1}{\lambda} \|\widetilde{f}\|_{L^2(I, \mu)} + \frac{1}{\lambda}\|b\|_{L^2(I, \mu)} \|\ell'\|_{L^{\infty}(I)}  + \frac{1}{\lambda} \|c\|_{L^2(I, \mu)} \|\ell\|_{L^{\infty}(I)}.
\end{align*}
Therefore, 
\begin{align*}
\|\widetilde{y} \|_{L^2(I, \mu)} &\leq \|\widetilde{y}-\ell \|_{L^2(I, \mu)} + \|\ell\|_{L^2(I, \mu)}  \leq  \|\widetilde{y}-\ell \|_{L^2(I, \mu)} + |I|^{1/2} (\max_{\overline{I}} \rho)^{1/2} \|\ell\|_{L^{\infty}(I)} \\
&\leq   \frac{1}{\lambda} \|\widetilde{f}\|_{L^2(I, \mu)} + \frac{1}{\lambda}\|b\|_{L^2(I, \mu)}  |I|^{-1} |p-q|  + \left( \frac{1}{\lambda} \|c\|_{L^2(I, \mu)}+|I|^{1/2}  (\max_{\overline{I}} \rho)^{1/2} \right) (|p| \vee |q|) \\
&\leq  \frac{1}{\lambda} \|\widetilde{f}\|_{L^2(I, \mu)}+     \left(\frac{1}{\lambda}\|b\|_{L^2(I, \mu)}  |I|^{-1}+ \frac{1}{\lambda} \|c\|_{L^2(I, \mu)}+|I|^{1/2}  (\max_{\overline{I}} \rho)^{1/2} \right) (|p| \vee |q|).
\end{align*}
The rest immediately follows from the above.
\end{proof}

\begin{theo} \label{contrastabilest}
Let $b \in L^2(I)$, $c \in L^2(I)$ with $c \geq \lambda$ on $I$ for some constant $\lambda>0$ and $\widetilde{f} \in L^2(I)$, where $I=(x_1, x_2)$ is a bounded open interval. Let
$\rho$ be a function defined in \eqref{defnrhofun}.
Let $\widetilde{y} \in H^{2,1}(I) \cap C^1(\overline{I})$ be the unique strong solution to
\eqref{bvpstrong} as in Theorem \ref{exunisol}. 
Let $\Phi \in H^{2,2}(I) \cap C^1(\overline{I})$. Then,
\begin{align*}
\| \widetilde{y}-\Phi \|_{L^2(I)} \leq \left(\frac{\max_{\overline{I}}\rho}{\min_{\overline{I}} \rho } \right)^{1/2} \frac{1}{\lambda} \big\| \widetilde{f} - \mathcal{L}[\Phi]   \big \|_{L^2(I)}+ \left(\frac{\max_{\overline{I}}\rho}{\min_{\overline{I}} \rho } \right)^{1/2}\widetilde{K}_3\bigg( \big|p-\Phi(x_1)\big|  \vee \big|q-\Phi(x_2)\big| \bigg),
\end{align*}
where $\widetilde{K}_3>0$ is a constant  as in Theorem \ref{contrastabiles} and $\mathcal{L}$ is a differential operator defined by \eqref{mathcalloper}.
\end{theo}
\noindent
\begin{proof}
Note that $\mathcal{L}[\widetilde{y}]=-\widetilde{y}''+b\widetilde{y}'+c\widetilde{y} = \widetilde{f} \in L^2(I)$ and $\mathcal{L}[\Phi] = -\varepsilon \Phi''+b \Phi'+c\Phi \in L^2(I)$. Let $u = \widetilde{y}-\Phi$. Then, $u \in H^{2,2}(I) \cap C^1(\overline{I})$ and $\mathcal{L}[u]=\widetilde{f}-\mathcal{L}[\Phi] \in L^2(I)$. Moreover, by Theorem \ref{basictheo} $u$ is a uniqe strong solution to the following problem:
\begin{equation*} \label{bvpstoequ}
\left\{
\begin{alignedat}{2}
& -\varepsilon u''+b u' +c u= \widetilde{f}-\mathcal{L}[\Phi] \quad \text{ on }  I, \\
&\; u(x_1)=p-\Phi(x_1), \quad u(x_2)=q-\Phi(x_2).
\end{alignedat} \right.
\end{equation*}
By Theorem \ref{contrastabiles},
$$
\|u\|_{L^2(I)} \leq \left(\frac{\max_{\overline{I}}\rho}{\min_{\overline{I}} \rho } \right)^{1/2} \frac{1}{\lambda} \| \widetilde{f}-\mathcal{L}[\Phi] \|_{L^2(I)} +\left(\frac{\max_{\overline{I}}\rho}{\min_{\overline{I}} \rho } \right)^{1/2}  \widetilde{K}_3\bigg( \big|p-\Phi(x_1)\big|  \vee \big|q-\Phi(x_2)\big| \bigg),
$$
as desired.
\end{proof}

\begin{theo} \label{contracproweig}
Let $\varepsilon>0$ be a constant, $b \in L^2(I)$, $c \in L^2(I)$ with $c \geq \lambda$ on $I$ for some constant $\lambda>0$ and $\widetilde{f} \in L^2(I)$, where $I=(x_1, x_2)$ is a bounded open interval. Let $p,q \in \mathbb{R}$ be given.
Let
$\rho$ be a function defined in \eqref{defnrhofun}.
Let $\widetilde{y} \in H^{2,1}(I) \cap C^1(\overline{I})$ be the unique strong solution to
\eqref{bvpstrong} as in Theorem \ref{exunisol}. Let $\Phi_A \in H^{1,2}_0(I) \cap H^{2,2}(I) \cap C^1(\overline{I})$ and $\ell(x):=\frac{q-p}{x_2-x_1}(x-x_1)+p$, $x \in \mathbb{R}$. Then,
\begin{align*}
\big\| \widetilde{y} - (\Phi_A +\ell)  \big\|_{L^2(I)} &\leq \left(\frac{\max_{\overline{I}}\rho}{\min_{\overline{I}} \rho } \right)^{1/2}\frac{1}{\lambda} \big\|\widetilde{f} -\mathcal{L}[\Phi_A+\ell]\big\|_{L^2(I)} \\
& \leq \exp\left(\int_{I} \frac{1}{2\varepsilon} |b| dx \right) \frac{1}{\lambda}  \big\|\widetilde{f} -\mathcal{L}[\Phi_A+\ell]\big\|_{L^2(I)}, 
\end{align*}
where $\mathcal{L}$ is a differential operator defined by \eqref{mathcalloper}.
\end{theo}

\noindent
\begin{proof}
Since $(\Phi_A +\ell)(x_1)=p$ and $(\Phi_A +\ell)(x_2)=q$, the assertion follows from Theorem \ref{contrastabilest} where 
$\Phi$ there is replaced by $\Phi_A+\ell$ here.
\end{proof}

\subsection{Numerical experiments for PINN with $L^2(I, \mu)$-contraction estimates}\label{weightsimul}
\noindent
Let us consider the conditions as in Section \ref{basicexperime} and additionally assume that $c \geq \lambda$ for some $\lambda>0$.
Then, by using Theorem \ref{contracproweig} with the Monte Carlo integration as in \eqref{montecarint}, we discover that for sufficiently large $n \in \mathbb{N}$
\begin{align}
\underbrace{\frac{1}{n} \sum_{i=1}^n |\widetilde{y}(X_i)-(\Phi_A+\ell)(X_i)|^2}_{=: {\sf Error}} &\leq \frac{1}{\lambda^2} \left(\frac{\max_{\overline{I}}\rho}{\min_{\overline{I}} \rho } \right) \cdot  \frac{1}{n} \sum_{i=1}^n \Big(\widetilde{f}(X_i)-\mathcal{L}[\Phi_A+\ell](X_i) \Big)^2, \nonumber \\
& \leq \frac{1}{\lambda^2} \exp\left(\int_{I} \frac{1}{\varepsilon} |b| dx \right)\cdot  \underbrace{\frac{1}{n} \sum_{i=1}^n \Big(\widetilde{f}(X_i)-\mathcal{L}[\Phi_A+\ell](X_i) \Big)^2}_{=:{\sf Loss}}, \quad \text{ very likely}. \label{weighterbdd}
\end{align}

\begin{exam} \label{exampleweight} \rm
Let $I=(0, 1)$. Assume that $b=2$, $p=0$, $q=1$ and $\widetilde{f}=0$. Let $\lambda>0$ be a constant and $c:=\lambda$. Then, \eqref{bvpstrong} is expressed as
\begin{equation} \label{bvpstrongspeci}
\left\{
\begin{alignedat}{2}
& -\varepsilon \widetilde{y}''+2 \widetilde{y}' +\lambda \widetilde{y}=  \widetilde{f} \quad \text{ on }  I, \\
&\widetilde{y}(0)=0, \quad \widetilde{y}(1)=1,
\end{alignedat} \right.
\end{equation}
which has a unique solution by Theorem \ref{exunisol}.
\noindent
All experimental details are identical to those described in Example \ref{pinn12newex}.
\begin{itemize}
\item[(1)]
Let $\varepsilon=1$.
Then, 
$$
\widetilde{y}(x):= \frac{1}{e^{1+\sqrt{1+\lambda}}-e^{1-\sqrt{1+\lambda}}}\, e^{(1+\sqrt{1+\lambda})x} + \frac{-1}{e^{1+\sqrt{1+\lambda}}-e^{1-\sqrt{1+\lambda}}} \, e^{(1-\sqrt{1+\lambda})x}, \quad x \in \mathbb{R}
$$
is a unique solution to \eqref{bvpstrongspeci}
\begin{figure}[!h]
     \centering
     \begin{subfigure}[b]{0.43\textwidth}
         \centering
         \includegraphics[width=\textwidth]{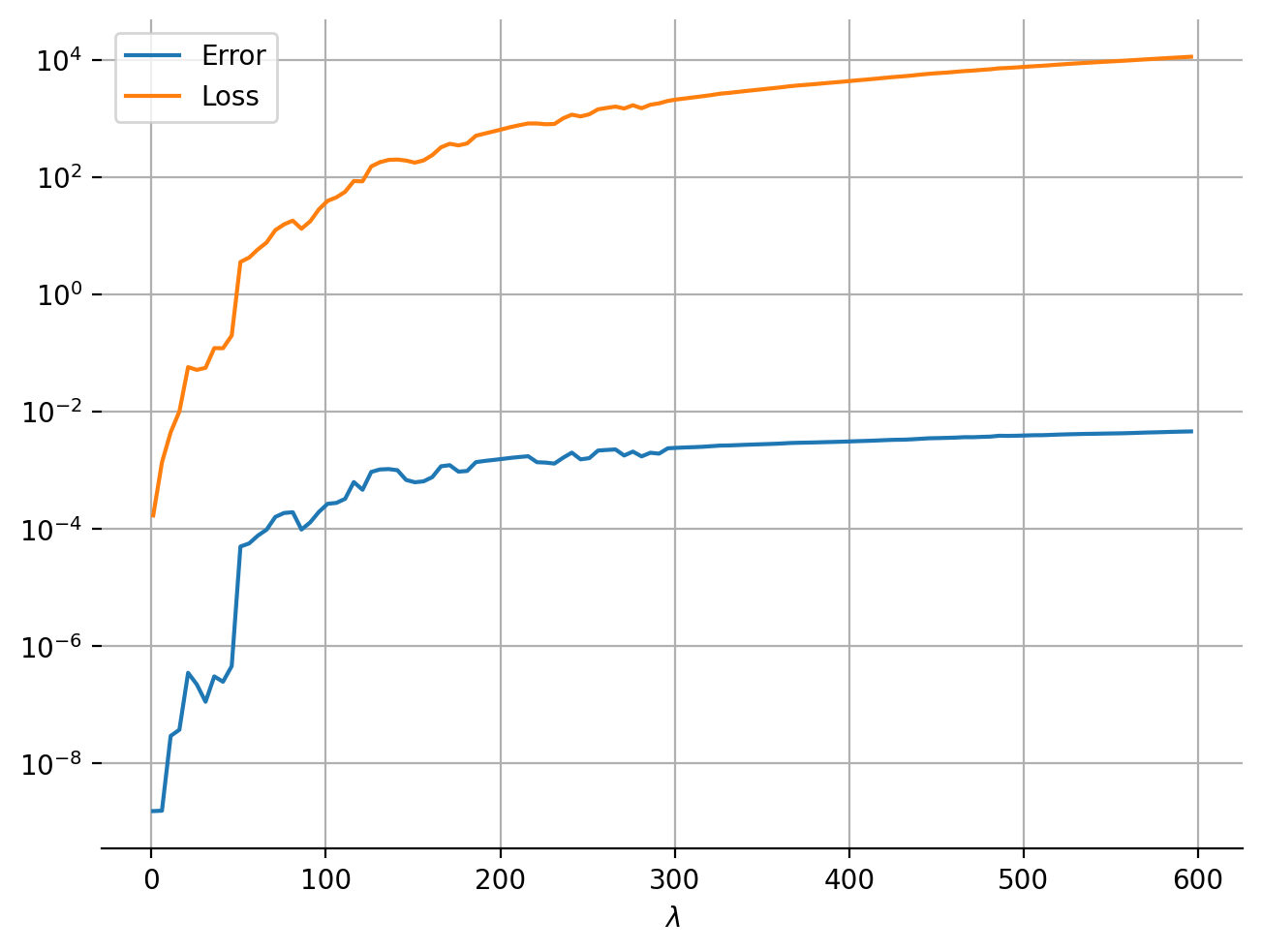}
	\caption{{\sf Loss} and {\sf Error} after 500 epochs as $\lambda$ increases}
         \label{fig:train}
     \end{subfigure}
     \hfill
     \begin{subfigure}[b]{0.43\textwidth}
         \centering
         \includegraphics[width=\textwidth]{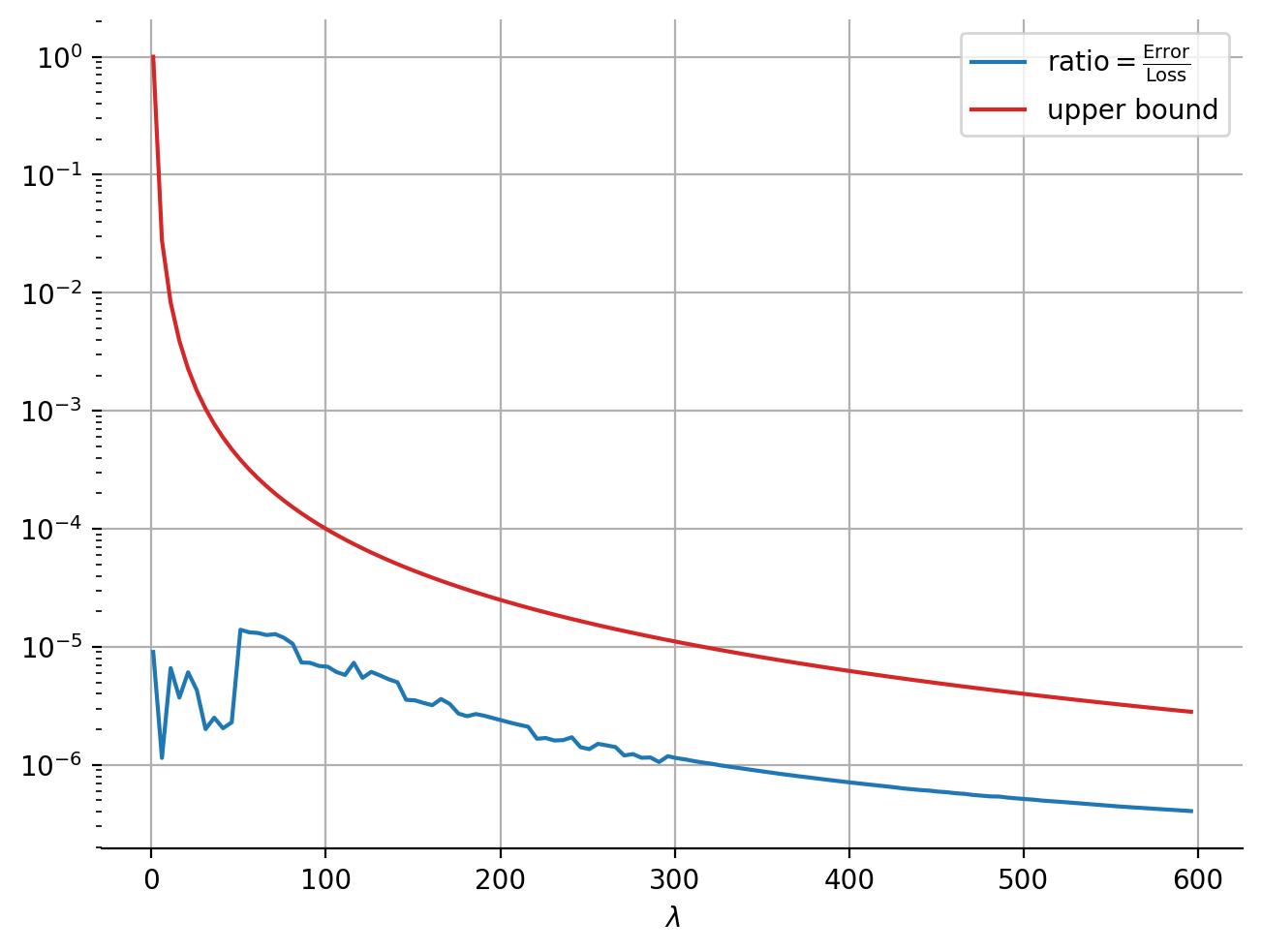}
	\caption{$\frac{\sf Error}{\sf Loss}$ bounded by the upper bound $\frac{1}{\lambda^2}$ as $\lambda$ increases}
         \label{fig:result}
     \end{subfigure}
     \caption{As $\lambda$ increases, {\sf Loss} increases but {\sf Error} remains robust due to the decrease of $\frac{\sf Error}{\sf Loss}$}
     \label{figweighestim}
\end{figure}
\end{itemize}
In Figure \ref{figweighestim}(a), we can observe that as $\lambda$ increases, {\sf Loss} increases whereas {\sf Error} is relatively stable. The main reason for this is explained by Figure \ref{figweighestim}(b) where $\frac{\sf Error}{\sf Loss}$ decreases with the upper bound $\frac{1}{\lambda^2}$ as $\lambda$ increases. In the error estimates \eqref{weighterbdd}, we derive the upper bound of $\frac{\sf Error}{\sf Loss}$ as
$\frac{1}{\lambda^2} \exp\left(\int_{I} \frac{1}{\varepsilon} |b| dx \right) = \frac{e^2}{\lambda^2}$.
However, by  Figure \ref{figweighestim}(b), 
one can see that $\frac{1}{\lambda^2}$ is more suitable upper bound for $\frac{\sf Error}{\sf Loss}$ than $\frac{e^2}{\lambda^2}$.
\begin{itemize}
\item[(2)]
For general $\varepsilon>0$, we obtain that
$$
\widetilde{y}(x):= \frac{1}{e^{\frac{1}{\varepsilon}+\sqrt{\frac{1}{\varepsilon^2}+\frac{\lambda}{\varepsilon}}}-e^{\frac{1}{\varepsilon}-\sqrt{\frac{1}{\varepsilon^2}+\frac{\lambda}{\varepsilon}}}} e^{\left(  \frac{1}{\varepsilon} + \sqrt{\frac{1}{\varepsilon^2} + \frac{\lambda}{\varepsilon}} \right) x}+ \frac{-1}{e^{\frac{1}{\varepsilon}+\sqrt{\frac{1}{\varepsilon^2}+\frac{\lambda}{\varepsilon}}}-e^{\frac{1}{\varepsilon}-\sqrt{\frac{1}{\varepsilon^2}+\frac{\lambda}{\varepsilon}}}}  e^{\left(  \frac{1}{\varepsilon} - \sqrt{\frac{1}{\varepsilon^2} + \frac{\lambda}{\varepsilon}} \right) x}, \quad x \in \mathbb{R}
$$
is a unique solution to \eqref{bvpstrongspeci} by Theorem \ref{exunisol}.
\begin{figure}[!h]
     \centering
     \begin{subfigure}[b]{0.43\textwidth}
         \centering
         \includegraphics[width=\textwidth]{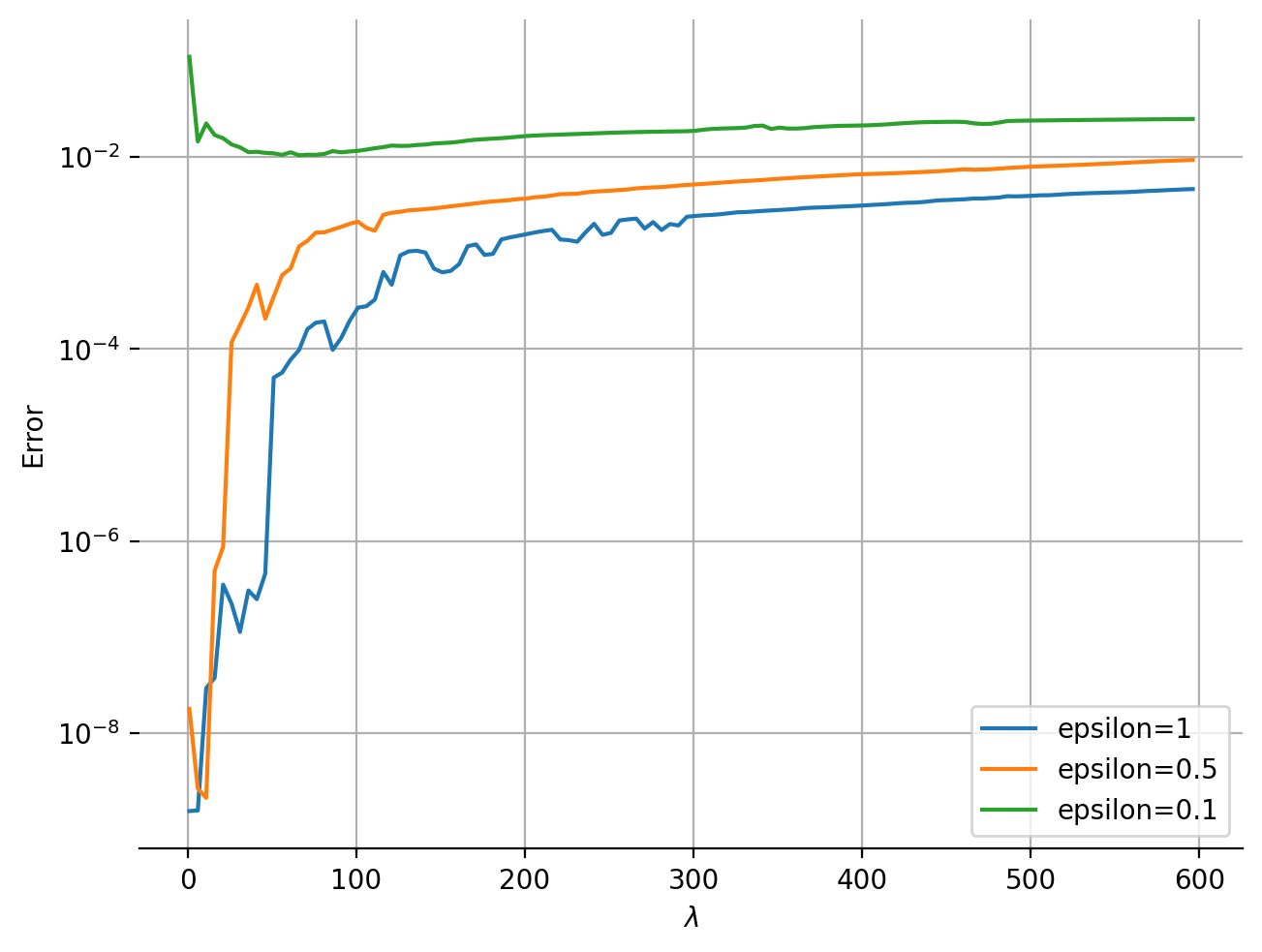}
	\caption{Comparison of {\sf Errors} after 500 epochs as $\lambda$ increase when $\varepsilon=1, 0.5, 0.1$}.
         \label{fig:train}
     \end{subfigure}
     \hfill
     \begin{subfigure}[b]{0.43\textwidth}
         \centering
         \includegraphics[width=\textwidth]{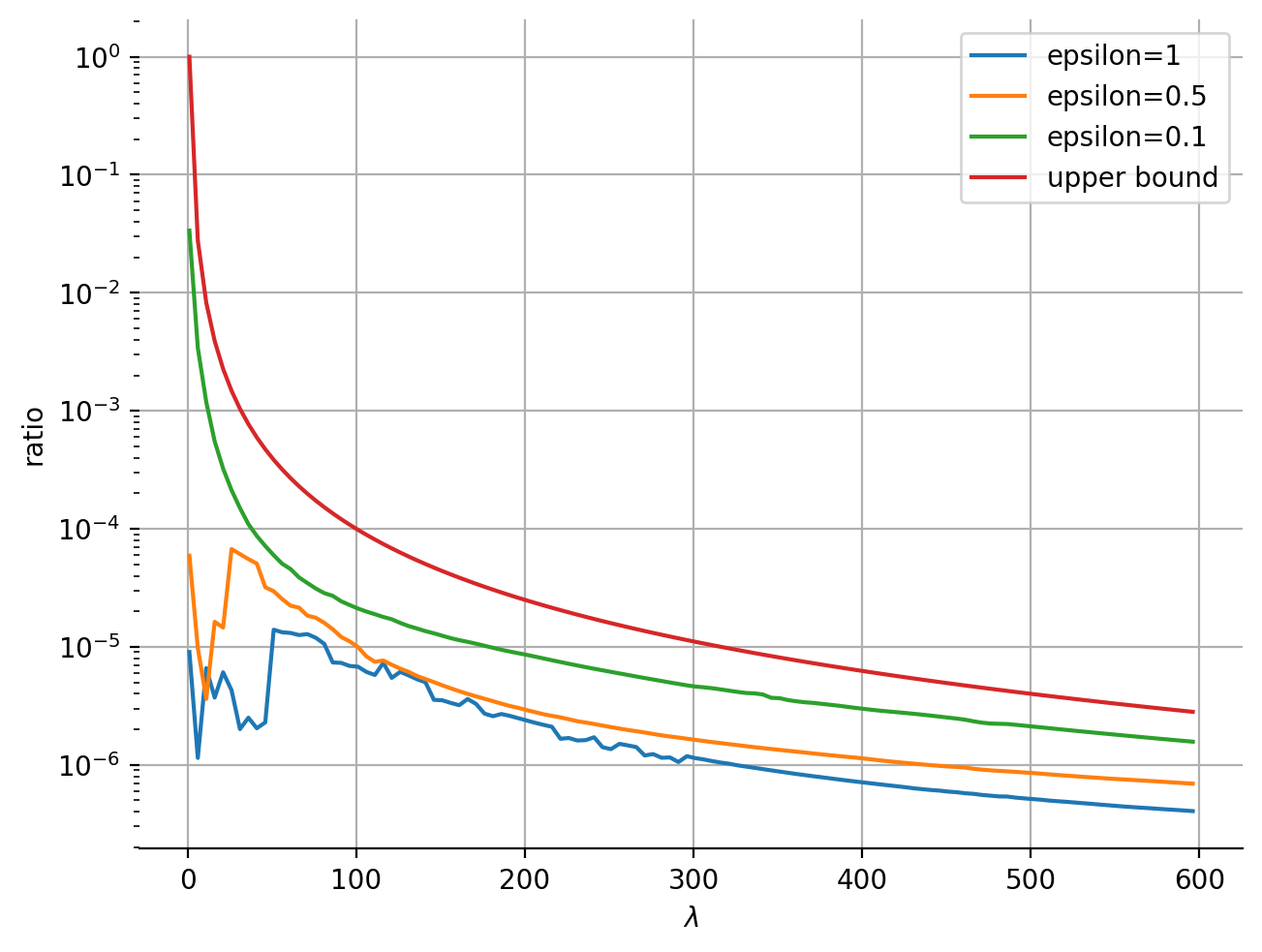}
	\caption{Comparison of the ratios ($\frac{\sf Error}{\sf Loss}$) as $\lambda$ increases when $\varepsilon=1, 0.5, 0.1$ and the upper bound $\frac{1}{\lambda^2}$}.
         \label{fig:result}
     \end{subfigure}
     \caption{The change of $\varepsilon$ does not significantly affect {\sf Error} behavior}
     \label{figchanepsil}
\end{figure}
\end{itemize}
Figure \ref{figchanepsil}(a) describes the behavior of {\sf Error} for each case when $\varepsilon=1, 0.5, 0.1$. In Figure \ref{figchanepsil}(a), if $\lambda$ is small, there are rapid changes of {\sf Error}, but if $\lambda$ is over $100$, {\sf Error} remains relatively stable. Figure \ref{figchanepsil}(b) describes the behaviors of the ratios ($\frac{\sf Error }{\sf Loss}$) which rapidly decrease as $\lambda$ increases. The rapid decrease of $\frac{\sf Error }{\sf Loss}$ as the increase of $\lambda$ in Figure \ref{figchanepsil}(b) explains the reason {\sf Error} remains relatively stable in Figure \ref{figchanepsil}(a). Here, we point out that the upper bound of $\frac{\sf Error}{\sf Loss}$ is $\frac{1}{\lambda^2}$ in Figure \ref{figchanepsil}(b). Indeed, if $\varepsilon=0.1$, then one can derive an upper bound of $\frac{\sf Error}{\sf Loss}$ as $\frac{1}{\lambda^2} \exp\left(\int_{I} \frac{1}{\varepsilon} |b| dx \right)=\frac{e^{20}}{\lambda^2}$ as in \eqref{weighterbdd}, which is much bigger than $\frac{1}{\lambda^2}$. Therefore, we have to derive sharper error estimates than \eqref{weighterbdd}, which will be presented in the next section.
\end{exam}

\section{Error analysis via $L^2(I, dx)$-contraction estimates } \label{l2dxcontrasec}
\subsection{$L^2(I, dx)$-contraction estimates }
\begin{theo}[$L^2(I, dx)$-contraction estimates] \label{contradxestim}
Let $\varepsilon>0$ be a constant, $b \in H^{1,1}(I)$, $c \in L^1(I)$ with $c \geq 0$ on $I$ and $f \in L^2(I)$, where $I$ is a bounded open interval.  Assume that there exists a constant $\gamma >0$ such that $-\frac{1}{2}b'+c \geq \gamma$ on $I$. Let $y \in H^{1,2}_0(I) \cap H^{2,1}(I) \cap C^1(\overline{I})$ be a unique strong solution to \eqref{strongequ}. Then,
\begin{equation} \label{contrarwrtdx}
\| y\|_{L^2(I)} \leq \frac{1}{\gamma} \|f\|_{L^2(I)}.
\end{equation}
\end{theo}
\begin{proof}
Multiplying both sides of \eqref{strongequ} by $y$ and integrating over $I$ with respect to $dx$, we have
\begin{align*}
\gamma \int_{I} y^2 dx &\leq \int_{I} \varepsilon |y'|^2 \rho dx + \int_{I} \Big(-\frac12 b'+c\Big)y^2 dx =\int_{I} \varepsilon |y'|^2 dx + \int_{y} b y' y dx + \int_{I} cy^2 dx\\\
&= \int_{I} f y dx \leq \left(\int_{I} f^2 dx \right)^{1/2} \left( \int_{I} y^2  dx \right)^{1/2},
\end{align*}
and hence the assertion follows.
\end{proof}

\begin{rem}
The above contraction estimates can be extended to the case of the multi-dimensional linear elliptic equations. 
As a general form of $L^2$-contraction estimates, one can derive $L^r$-contraction estimates with $r \in [1, \infty]$ by using the sub-Markovian property and the Riesz-Thorin interpolation with a duality argument (see \cite{L24}).
\end{rem}

\centerline{}

\begin{theo} \label{contrastabilesdx}
 Let $\varepsilon>0$ be a constant, $b \in H^{1,1}(I)$, $c \in L^2(I)$ with $c \geq 0$ on $I$  and $\widetilde{f} \in L^2(I)$, where $I$ is a bounded open interval.  Assume that there exists a constant $\gamma >0$ such that $-\frac{1}{2}b'+c \geq \gamma$ on $I$.
Then, the  unique strong solution $\widetilde{y} \in H^{2,1}(I) \cap C^1(\overline{I})$ to \eqref{bvpstrong} as in 
Theorem \ref{exunisol} satisfies
$$
\|\widetilde{y}\|_{L^2(I)} \leq \frac{1}{\gamma} \| \widetilde{f} \|_{L^2(I)} +\widetilde{K}_4 (|p| \vee |q|),
$$
where $\widetilde{K}_4:=\frac{1}{\gamma}\|b\|_{L^2(I)}  |I|^{-1}+ \frac{1}{\gamma} \|c\|_{L^2(I)}+|I|^{1/2} $.
\end{theo}
\begin{proof}
Let $\ell(x):=\frac{q-p}{x_2-x_1}(x-x_1)+p$, $x \in \mathbb{R}$. Then, we can verify that
$\|\ell\|_{L^{\infty}(I)} \leq |p|\vee|q|$ and $\|\ell'\|_{L^{\infty}(I)} \leq |I|^{-1} |p-q|$. Let $y:=\widetilde{y}-\ell$. Then, by Theorem \ref{exunisol}
$y \in H^{1,2}_0(I) \cap H^{2,1}(I) \cap C^1(\overline{I})$ is a unique strong solution to \eqref{strongequ2} where $f$ is replaced by $\widetilde{f}- b \ell' -c\ell$, and hence it follows from \eqref{contrarwrtdx} in Theorem \ref{contradxestim} that
\begin{align*}
\| \widetilde{y}-\ell \|_{L^2(I)}&= \|y\|_{L^2(I)} \leq \frac{1}{\gamma} \big\|\widetilde{f}- b \ell' -c\ell \big\|_{L^2(I)}\\
&\leq \frac{1}{\gamma} \|\widetilde{f}\|_{L^2(I)} + \frac{1}{\gamma}\|b\|_{L^2(I)} \|\ell'\|_{L^{\infty}(I)}  + \frac{1}{\gamma} \|c\|_{L^2(I)} \|\ell\|_{L^{\infty}(I)}.
\end{align*}
Therefore, 
\begin{align*}
\|\widetilde{y} \|_{L^2(I)} &\leq \|\widetilde{y}-\ell \|_{L^2(I)} + \|\ell\|_{L^2(I)}  \leq  \|\widetilde{y}-\ell \|_{L^2(I, \mu)} + |I|^{1/2}  \|\ell\|_{L^{\infty}(I)} \\
&\leq   \frac{1}{\gamma} \|\widetilde{f}\|_{L^2(I)} + \frac{1}{\gamma}\|b\|_{L^2(I)}  |I|^{-1} |p-q|  + \left( \frac{1}{\gamma} \|c\|_{L^2(I)}+|I|^{1/2}  \right) (|p| \vee |q|) \\
&\leq  \frac{1}{\gamma} \|\widetilde{f}\|_{L^2(I)}+     \left(\frac{1}{\gamma}\|b\|_{L^2(I)}  |I|^{-1}+ \frac{1}{\gamma} \|c\|_{L^2(I)}+|I|^{1/2}  \right) (|p| \vee |q|),
\end{align*}
as desired.
\end{proof}

\begin{theo}  \label{contierresdx}
Let $\varepsilon>0$ be a constant, $b \in H^{1,1}(I)$, $c \in L^2(I)$ with $c \geq 0$ on $I$ and $\widetilde{f} \in L^2(I)$, where $I=(x_1, x_2)$ is a bounded open interval. Assume that there exists a constant $\gamma >0$ such that $-\frac{1}{2}b'+c \geq \gamma$ on $I$. Let $p,q \in \mathbb{R}$ be given and $\widetilde{y} \in H^{2,1}(I) \cap C^1(\overline{I})$ be the unique strong solution to
\eqref{bvpstrong} as in Theorem \ref{exunisol}. 
Let $\Phi \in H^{2,2}(I) \cap C^1(\overline{I})$. Then,
\begin{align*}
\| \widetilde{y}-\Phi \|_{L^2(I)} \leq \frac{1}{\gamma} \big\| \widetilde{f} - \mathcal{L}[\Phi]   \big \|_{L^2(I)}+ \widetilde{K}_4\bigg( \big|p-\Phi(x_1)\big|  \vee \big|q-\Phi(x_2)\big| \bigg),
\end{align*}
where $\widetilde{K}_4>0$ is a constant in Theorem \ref{contrastabilesdx} and $\mathcal{L}$ is a differential operator defined as in \eqref{mathcalloper}.
\end{theo}
\noindent
\begin{proof}
Let $u := \widetilde{y}-\Phi$. 
Then, $u \in H^{2,2}(I) \cap C^1(\overline{I})$ and $\mathcal{L}[\Phi] \in L^2(I)$. Moreover, by Theorem \ref{exunisol} $u$ is a (unique) strong solution to the following problem:
\begin{equation*} \label{bvpstrong3dx}
\left\{
\begin{alignedat}{2}
& -\varepsilon u''+b u' +c u= \widetilde{f}-\mathcal{L}[\Phi] \quad \text{ on }  I, \\
&\; u(x_1)=p-\Phi(x_1), \quad u(x_2)=q-\Phi(x_2).
\end{alignedat} \right.
\end{equation*}
By Theorem \ref{exunisol}, the result follows.
\end{proof}

\begin{theo} \label{maincontradxes}
Let $\varepsilon>0$ be a constant, $b \in H^{1,1}(I)$, $c \in L^2(I)$ with $c \geq 0$ on $I$ and $\widetilde{f} \in L^2(I)$, where $I=(x_1, x_2)$ is a bounded open interval. Assume that there exists a constant $\gamma >0$ such that $-\frac{1}{2}b'+c \geq \gamma$ on $I$. Let $p,q \in \mathbb{R}$ be given and $\widetilde{y} \in H^{2,1}(I) \cap C^1(\overline{I})$ be the unique strong solution to
\eqref{bvpstrong} as in Theorem \ref{exunisol}.  Let $\Phi_A \in H^{1,2}_0(I) \cap H^{2,2}(I) \cap C^1(\overline{I})$ and $\ell(x):=\frac{q-p}{x_2-x_1}(x-x_1)+p$, $x \in \mathbb{R}$. Then,
\begin{align*}
\big\| \widetilde{y} - (\Phi_A +\ell)  \big\|_{L^2(I)} \leq \frac{1}{\gamma} \big\|\widetilde{f} -\mathcal{L}[\Phi_A+\ell]\big\|_{L^2(I)},
\end{align*}
where $\mathcal{L}$ is a differential operator defined by \eqref{mathcalloper}.
\end{theo}
\begin{proof}
Since $(\Phi_A +\ell)(x_1)=p$ and $(\Phi_A +\ell)(x_2)=q$, the assertion follows from Theorem \ref{contierresdx} where 
$\Phi$ there is replaced by $\Phi_A+\ell$ here.
\end{proof}
\subsection{Numerical experiments for PINN with $L^2(I, dx)$-contraction estimates } \label{mainerroranaly}
\noindent
Let us consider the conditions in Section \ref{basicexperime} and also assume that $b \in H^{1,1}(I)$, $c \in L^2(I)$ with $c \geq 0$ on $I$  and $f \in L^2(I)$, where $I$ is a bounded open interval. Additionally, assume that 
\begin{equation} \label{maincondiforour}
-\frac{1}{2}b'+c \geq \gamma, \quad \text{for some $\gamma>0$}.
\end{equation}
Then, by using Theorem \ref{maincontradxes} with the Monte Carlo integration as in \eqref{montecarint}, we obtain that for sufficiently large $n \in \mathbb{N}$
\begin{align} \label{errorestindis}
\underbrace{\frac{1}{n} \sum_{i=1}^n |\widetilde{y}(X_i)-(\Phi_A+\ell)(X_i)|^2}_{=:{\sf Error}} \leq \frac{1}{\gamma^2}  \cdot  \underbrace{\frac{1}{n} \sum_{i=1}^n \Big(\widetilde{f}(X_i)-\mathcal{L}[\Phi_A+\ell](X_i) \Big)^2}_{=:{\sf Loss}}, \quad \text{ very likely}.
\end{align}
\centerline{}

\begin{exam} \label{examplefirst} \rm
Let $I=(0,1)$. Let $\lambda, k \geq 0$ be constants. Let $b(x)=-kx$, $c(x)=\lambda$ and \\
$\widetilde{y}(x) =x(1-x)e^x+x^2$, $x \in \mathbb{R}$. Then,
$$
\widetilde{y}'(x) = (1-x-x^2)e^x+2x, \;\;\;\;  \widetilde{y}''(x)=(-3x-x^2)e^x+2, \quad \text{ for all $x \in \mathbb{R}$}.
$$
Let $f(x):=-\varepsilon \widetilde{y}''(x)+b(x)\widetilde{y}'(x)+c(x)\widetilde{y}(x)=\Big( kx^3 +(k-\lambda +\varepsilon)x^2 +(3\varepsilon-k+\lambda)x  \Big) e^x + \lambda x^2-2kx^2 -2\varepsilon$, $x \in \mathbb{R}$.
Then, $\widetilde{y}$ is a unique strong solution to \eqref{bvpstrong} with $p=0$ and $q=1$. Moreover, we can choose
$$
\gamma=\frac12 k + \lambda=-\frac12 b' + c.
$$
Therefore, the upper bound of $\frac{\sf Error}{\sf Loss}$ is expressed as
$$
\frac{1}{\gamma^2} =\frac{4}{(k+2\lambda)^2}.
$$
 All experimental details are identical to those described in Example \ref{pinn12newex}.
\begin{itemize}
\item[(1)]
Let $\varepsilon=1$ and $k=7$.
\begin{figure}[!h]
     \centering
     \begin{subfigure}[b]{0.43\textwidth}
         \centering
         \includegraphics[width=\textwidth]{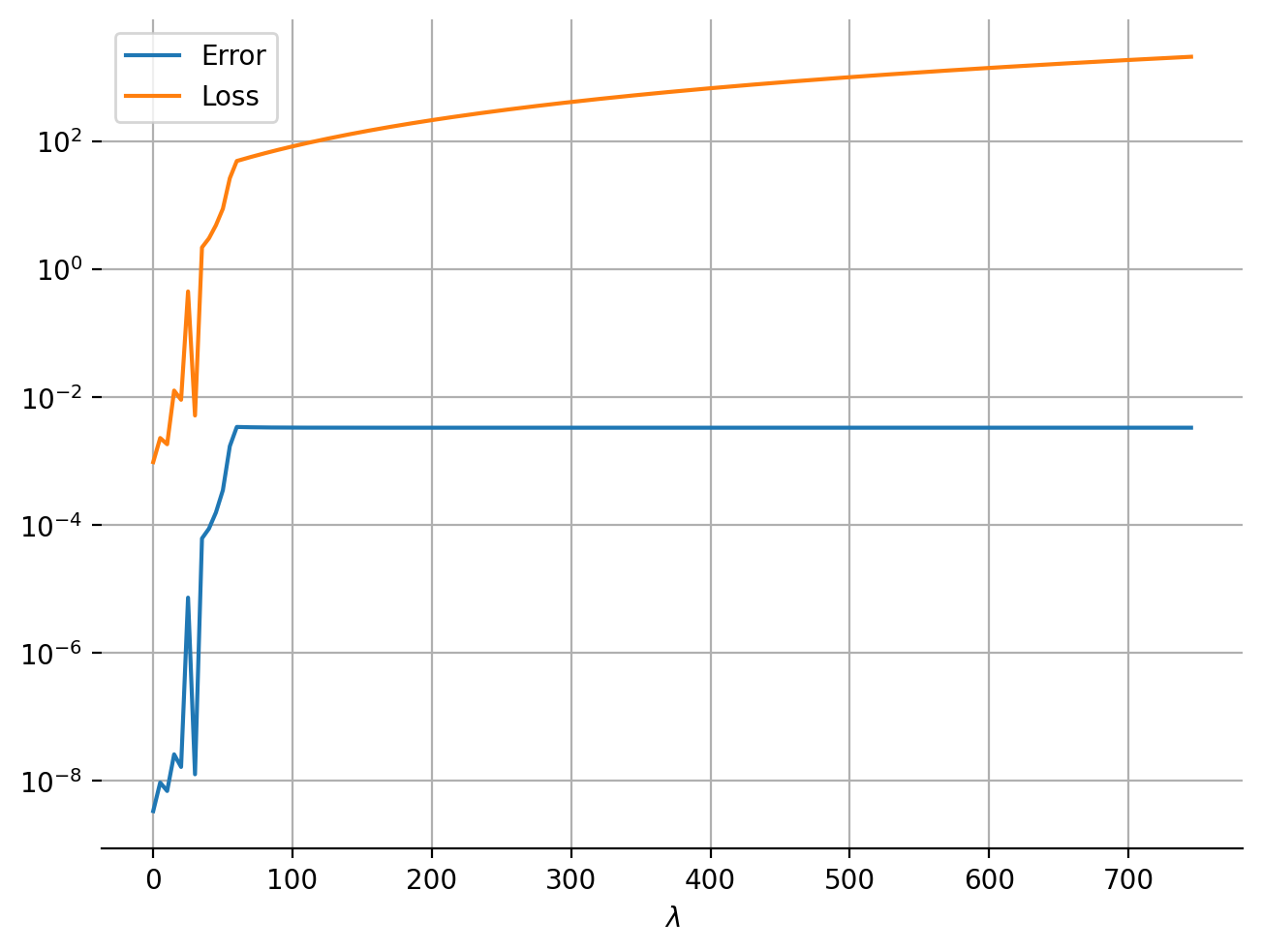}
	\caption{{\sf Loss} and {\sf Error} after 500 epochs as $\lambda$ increases \;($\varepsilon=1$, $k=7$)}
         \label{fig:train}
     \end{subfigure}
     \hfill
     \begin{subfigure}[b]{0.43\textwidth}
         \centering
\includegraphics[width=\textwidth]{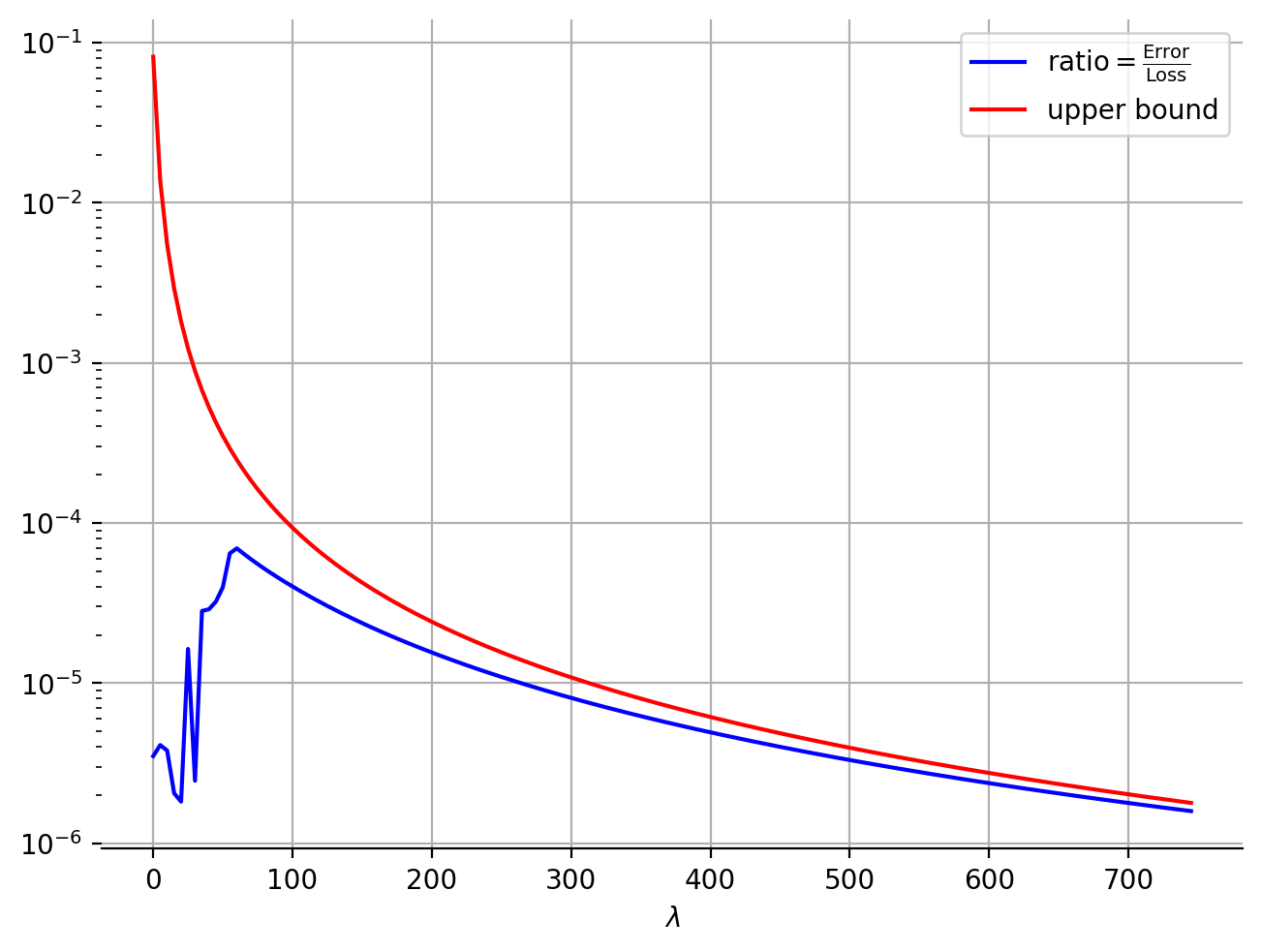}
	\caption{Ratio ($\frac{\sf Error}{\sf Loss}$) bounded by the upper bound $\frac{1}{\gamma^2}=\frac{4}{(7+2\lambda)^2}$\; ($\varepsilon=1$, $k=7$) }
         \label{fig:result}
     \end{subfigure}
     \caption{As $\lambda$ increases, {\sf Loss} increases but {\sf Error} remains robust (Example \ref{examplefirst})}
	\label{figdxcontraes}
\end{figure}
\end{itemize}
In Figure \ref{figdxcontraes}(a), we can observe that as $\lambda$ increases, {\sf Loss} also increases but {\sf Error} remains stable. This phenomenon is well explained by Figure \ref{figdxcontraes}(b) where the ratio ($\frac{\sf Error}{\sf Loss}$) rapidly decreases with the upper bound $\frac{1}{\gamma^2}=\frac{4}{(7+2\lambda)^2}$ as $\lambda$ increases. Furthermore, it is observed in Figure \ref{figdxcontraes}(b) that the upper bound of $\frac{\sf Error}{\sf Loss}$ is very sharp when $\lambda$ is large enough.

\newpage
\begin{itemize}
\item[(2)]
Let $\varepsilon=1$ and $\lambda=7$.
\begin{figure}[!h]
     \centering
     \begin{subfigure}[b]{0.43\textwidth}
         \centering
         \includegraphics[width=\textwidth]{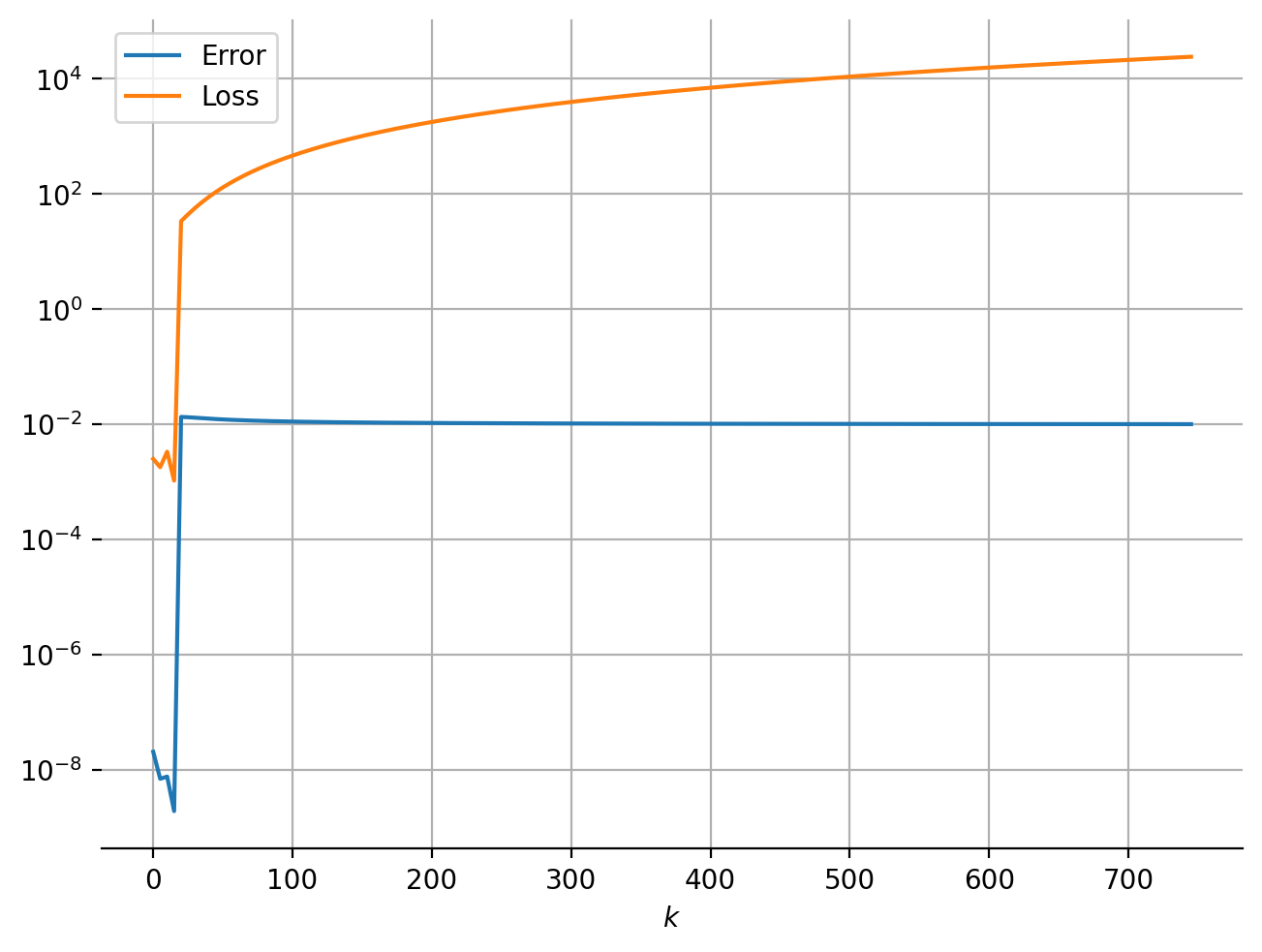}
	\caption{{\sf Loss} and {\sf Error} after 500 epochs as $k$ increases \; ($\varepsilon=1$, $\lambda=7$)}
         \label{fig:train}
     \end{subfigure}
     \hfill
     \begin{subfigure}[b]{0.43\textwidth}
         \centering
         \includegraphics[width=\textwidth]{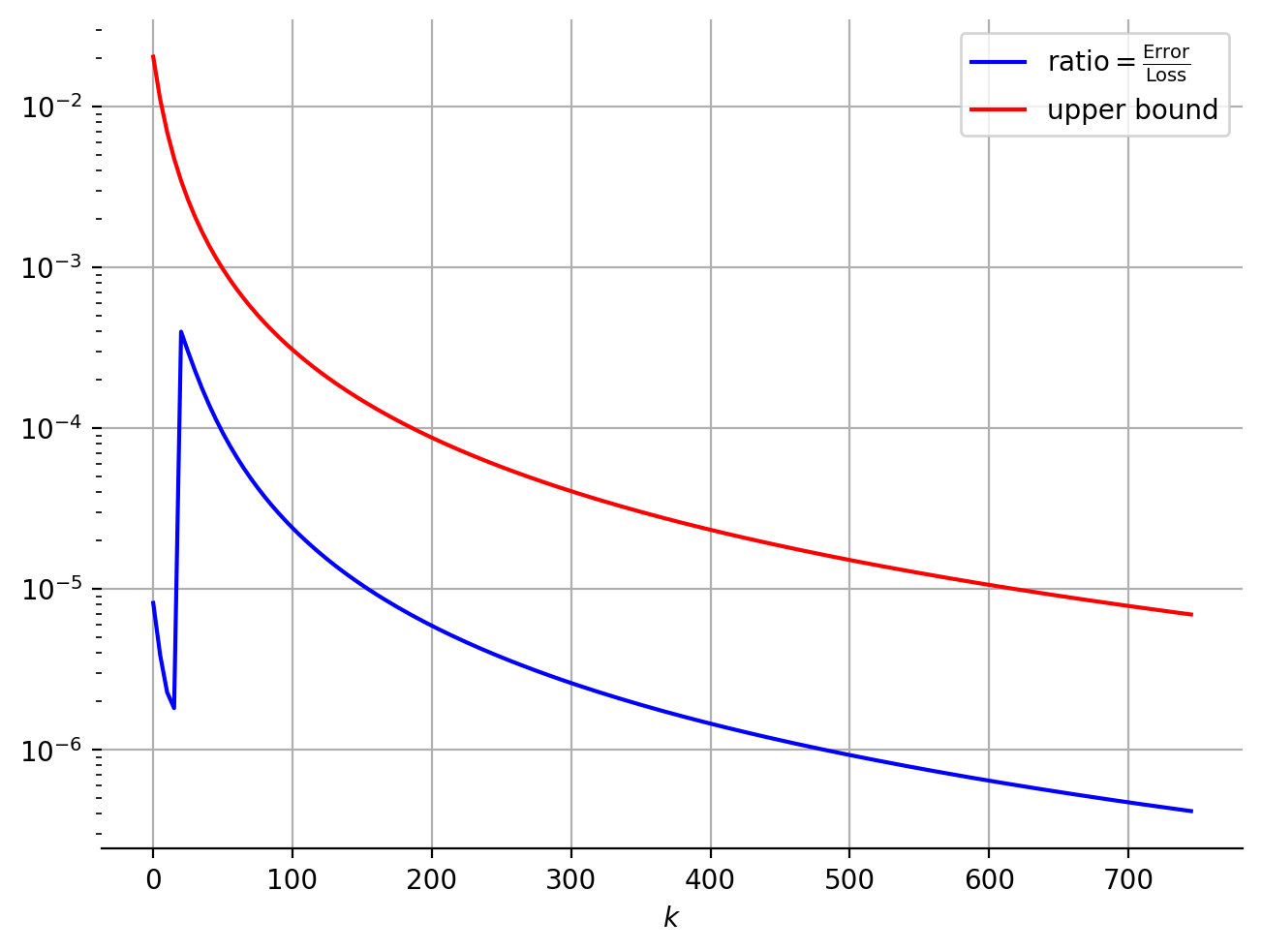}
	\caption{Ratio  ($\frac{\sf Error}{\sf Loss}$)  bounded by the upper bound $\frac{1}{\gamma^2} = \frac{4}{(k+14)^2}$ as $k$ increases \; ($\varepsilon=1$, $\lambda=7$)}
         \label{fig:result}
     \end{subfigure}
     \caption{As $k$ increases, {\sf Loss} increases but {\sf Error} remains robust (Example \ref{examplefirst})}
     \label{figaskincrero}
\end{figure}
\end{itemize}
In Figure \ref{figaskincrero}(a), we can observe that as $\lambda$ increases, {\sf Loss} also increases but {\sf Error} remains stable. This phenomenon is well explained by Figure \ref{figaskincrero}(b) where the ratio ($\frac{\sf Error}{\sf Loss}$) rapidly decreases with the upper bound $\frac{1}{\gamma^2}=\frac{4}{(k+14)^2}$ as $k$ increases.
\centerline{}
\centerline{}
\begin{itemize}
\item[(3)]
Let $k=10$ and $\lambda=15$.
\begin{figure}[!h]
     \centering
     \begin{subfigure}[b]{0.43\textwidth}
         \centering
         \includegraphics[width=\textwidth]{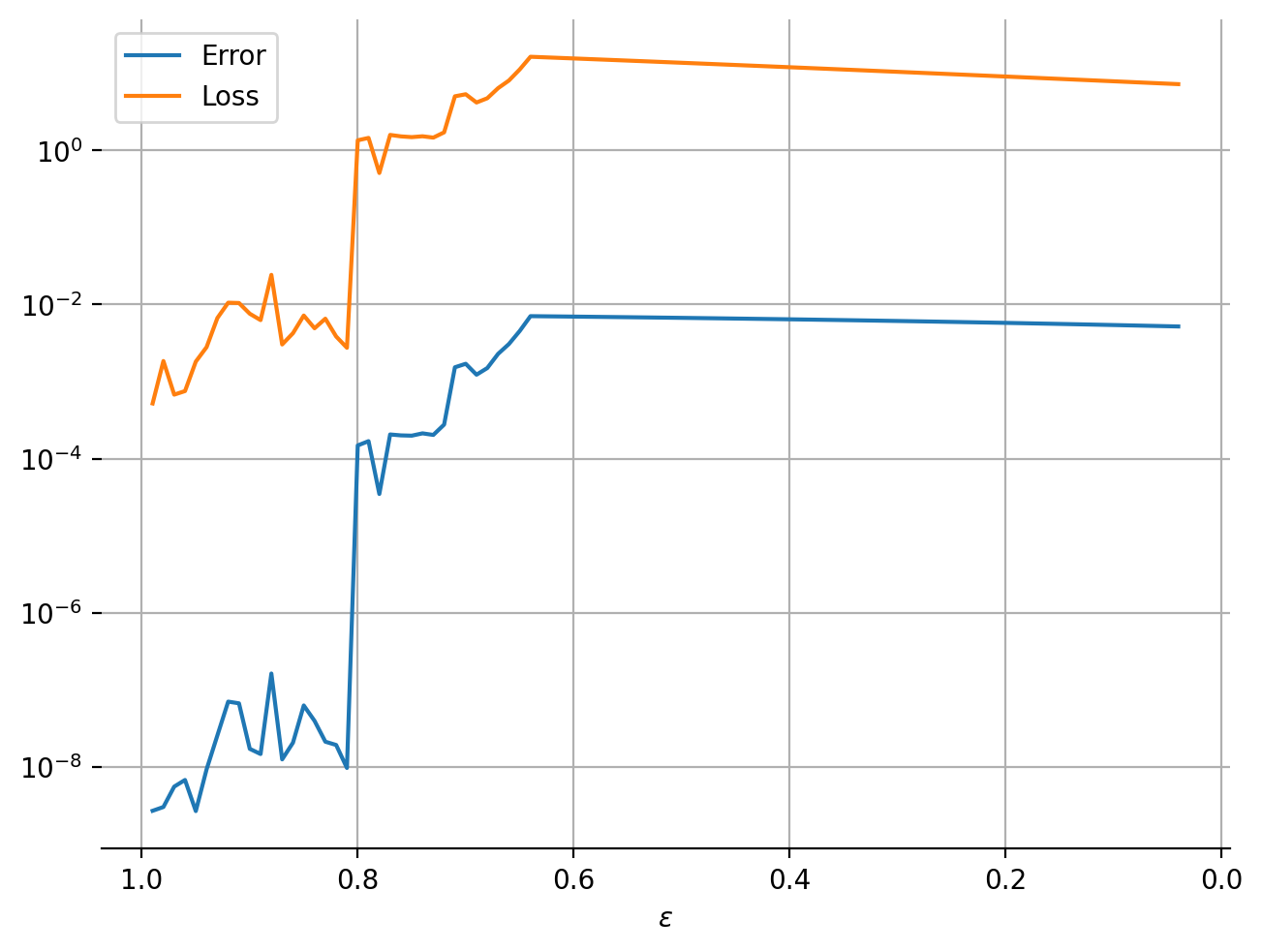}
	\caption{{\sf Loss} and {\sf Error} after 500 epochs as $\varepsilon$ decreases \; ($k=10$ and $\lambda=15$)}
         \label{fig:train}
     \end{subfigure}
     \hfill
     \begin{subfigure}[b]{0.43\textwidth}
         \centering
         \includegraphics[width=\textwidth]{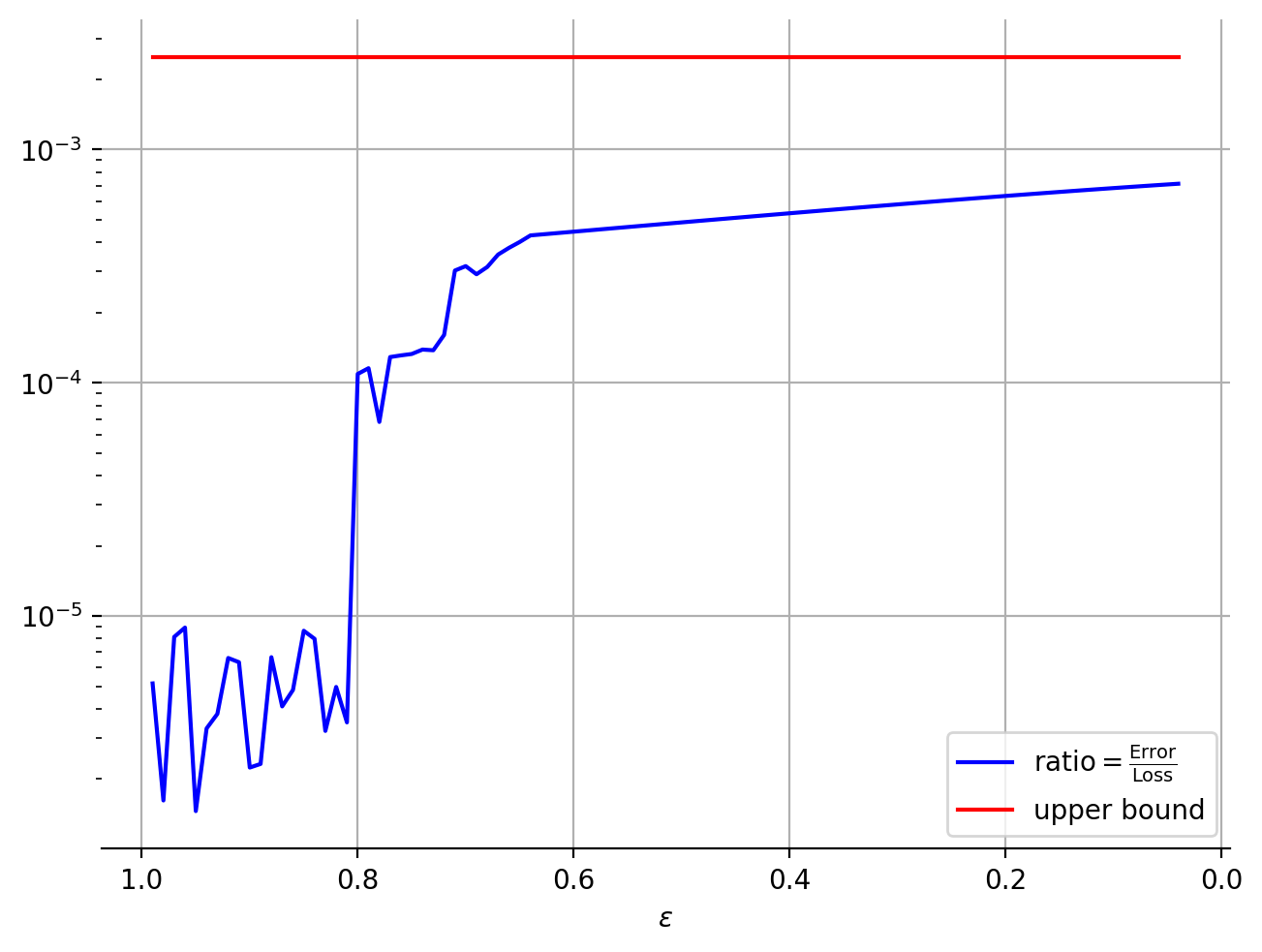}
	\caption{Ratio ($\frac{\sf Error}{\sf Loss}$) bounded by the upper bound $\frac{1}{\gamma^2}=\frac{1}{400}$ as $\varepsilon$ decreases \; ($k=10$ and $\lambda=15$)}
         \label{fig:result}
     \end{subfigure}
     \caption{As $\varepsilon$ decreases, {\sf Loss} and {\sf Error} remain stable (Example \ref{examplefirst})}
	\label{figasvarepincrero}
\end{figure}
\end{itemize}
	In Figure \ref{figasvarepincrero}(a), we can observe that as $\varepsilon$ decreases, {\sf Loss} and {\sf Error} remain robust. Since $\varepsilon$ is a very small value, the stability of {\sf Loss} is well-explained. Since $k$ and $\lambda$ are fixed as $10$ and $15$, respectively and $\varepsilon$ only decreases, we can expect that by our error estimates \eqref{errorestindis}, $\frac{\sf Error}{\sf Loss}$ does not change dramatically. This expectation is realized by Figure \ref{figasvarepincrero}(b) where $\frac{\sf Error}{\sf Loss}$ are bounded by $\frac{1}{\gamma^2}=\frac{1}{400}$, and hence the robustness of {\sf Error} is well-explained.
\end{exam}
\centerline{}
\centerline{}
\begin{exam} \label{example2nd} \rm
Let $I=(0,1)$. Let $\lambda, k \geq 0$ be constants. Let $b(x)=-kx$, $c(x)=\lambda$ and \\
$\widetilde{y}(x) =\sin \pi x + \cos \pi x$, $x \in \mathbb{R}$. Then,
$$
\widetilde{y}'(x) = \pi \cos \pi x - \pi \sin \pi x, \;\;\;\;\;  \widetilde{y}''(x)=-\pi^2 \sin \pi x -\pi^2 \cos \pi x, \quad \text{ for all $x \in \mathbb{R}$}.
$$
Let $f(x):=-\varepsilon \widetilde{y}''(x)+b(x)\widetilde{y}'(x)+c\widetilde{y}(x)=\Big(   \varepsilon \pi^2 +\pi k x +\lambda \Big)  \sin \pi x + \Big( \varepsilon \pi^2 -\pi k x +\lambda \Big) \cos \pi x$, \,$x \in \mathbb{R}$.
Then, $\widetilde{y}$ is a unique strong solution to \eqref{bvpstrong} with $p=1$, $q=-1$. As in Example \ref{examplefirst}, the upper bound of $\frac{\sf Error}{\sf Loss}$ is expressed as
$\frac{1}{\gamma^2} =\frac{4}{(k+2\lambda)^2}$.
\\ \\
All experimental details are identical to those described in Example \ref{pinn12newex}. \\
\begin{itemize}
\item[(1)]
Let $\varepsilon=1$ and $k=7$.
\begin{figure}[!h]
     \centering
     \begin{subfigure}[b]{0.43\textwidth}
         \centering
         \includegraphics[width=\textwidth]{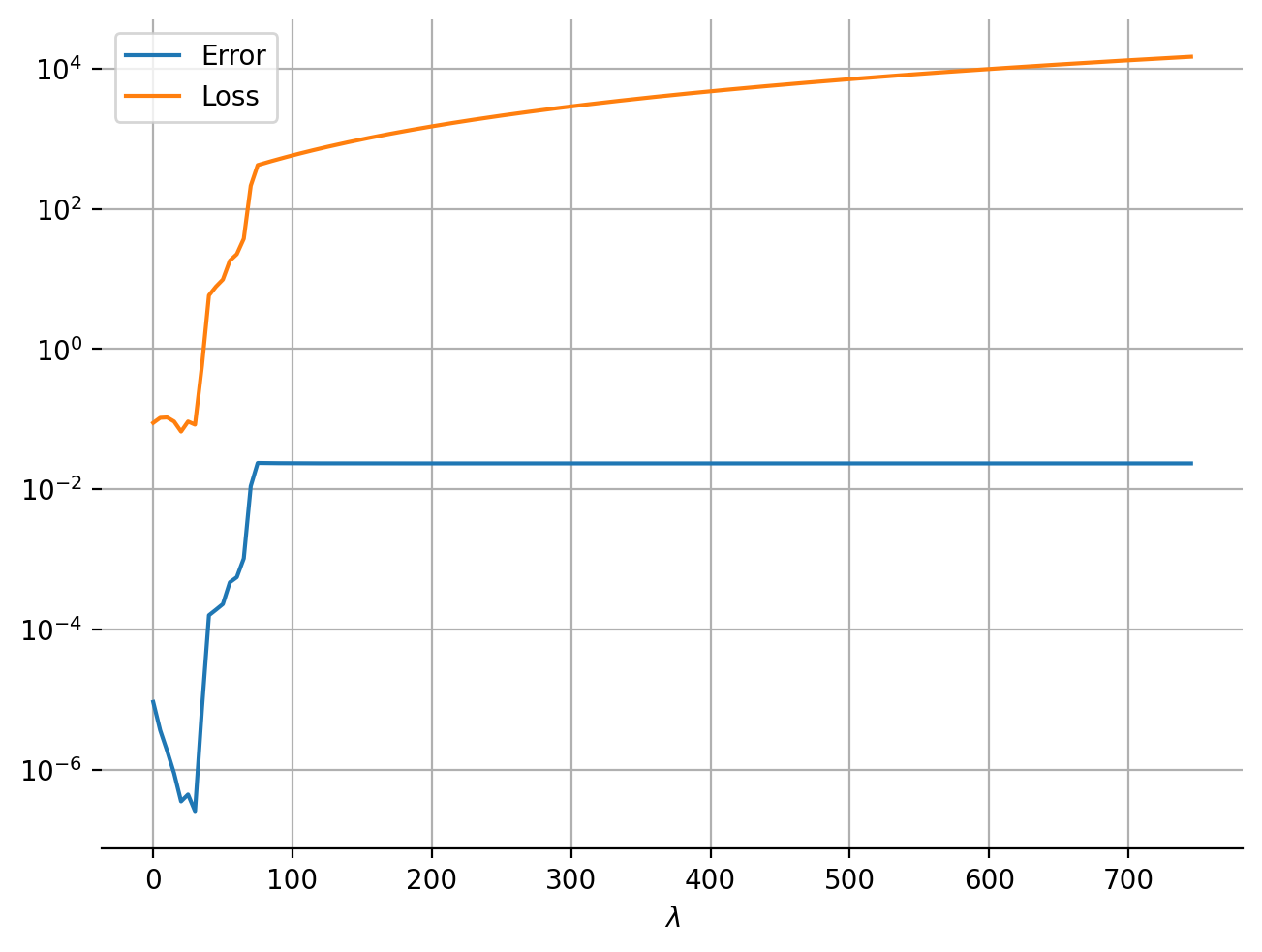}
	\caption{{\sf Loss} and {\sf Error} after 500 epochs as $\lambda$ increases \; ($\varepsilon=1$, $k=7$)}
         \label{fig:train}
     \end{subfigure}
     \hfill
     \begin{subfigure}[b]{0.43\textwidth}
         \centering
         \includegraphics[width=\textwidth]{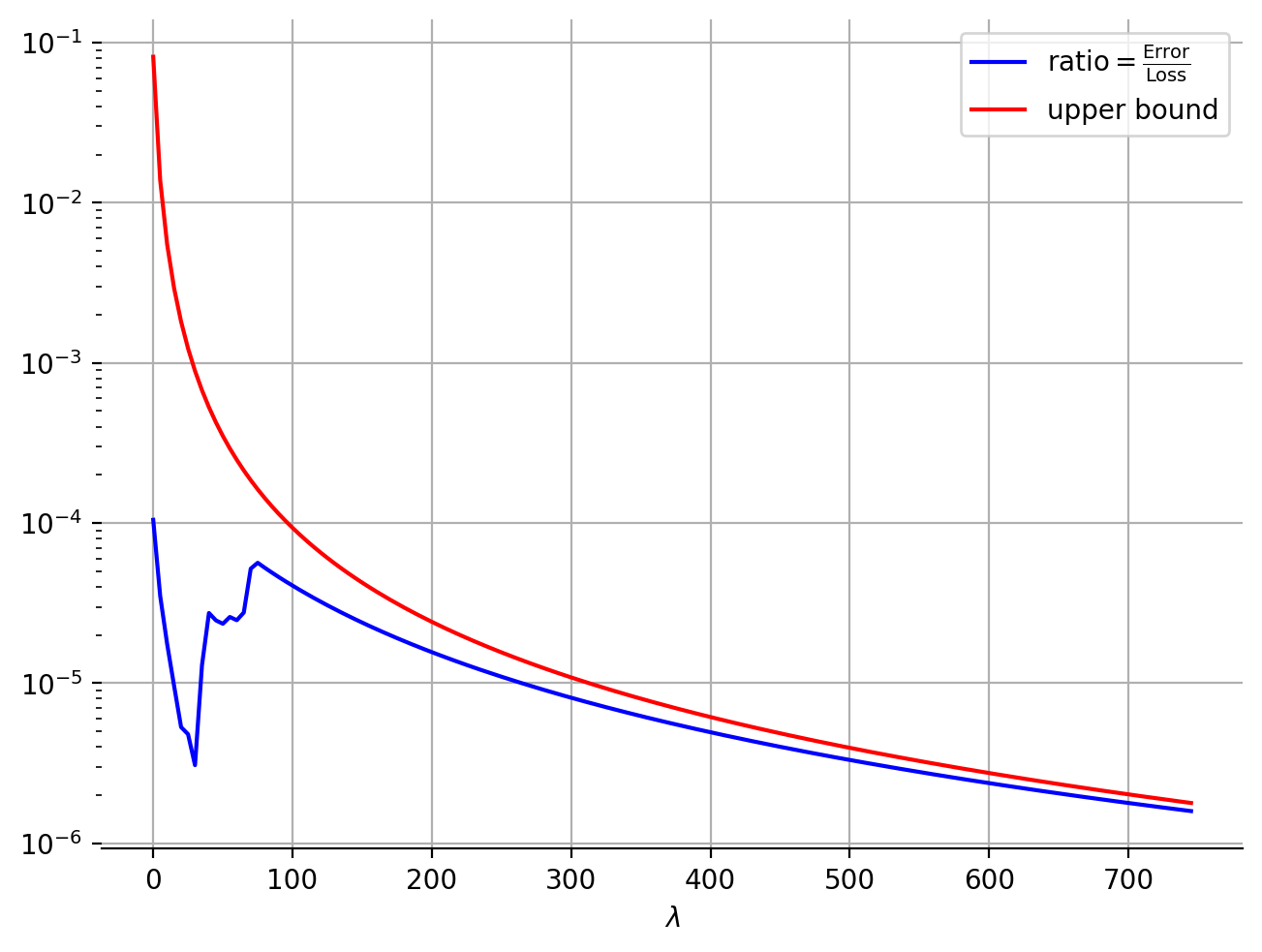}
	\caption{Ratio ($\frac{\sf Error}{\sf Loss}$) bounded by the upper bound $\frac{1}{\gamma^2}=\frac{4}{(7+2\lambda)^2}$ as $\lambda$ increases \;($\varepsilon=1$, $k=7$)}
     \end{subfigure}
\caption{As $\lambda$ increases, {\sf Loss} increases but {\sf Error} remains robust (Example \ref{example2nd})}
 \label{figlastlambda}
\end{figure}
\end{itemize}
In Figure \ref{figlastlambda}(a), we can observe that as $\lambda$ increases, {\sf Loss} also increases but {\sf Error} remains stable. This phenomenon is well explained by Figure \ref{figlastlambda}(b) where the ratio ($\frac{\sf Error}{\sf Loss}$) rapidly decreases with the upper bound $\frac{1}{\gamma^2}=\frac{4}{(7+2\lambda)^2}$ as $\lambda$ increases.
\centerline{}
\newpage
\begin{itemize}
\item[(2)]
Let $\varepsilon=1$ and $\lambda=7$.
\begin{figure}[!h]
     \centering
     \begin{subfigure}[b]{0.43\textwidth}
         \centering
         \includegraphics[width=\textwidth]{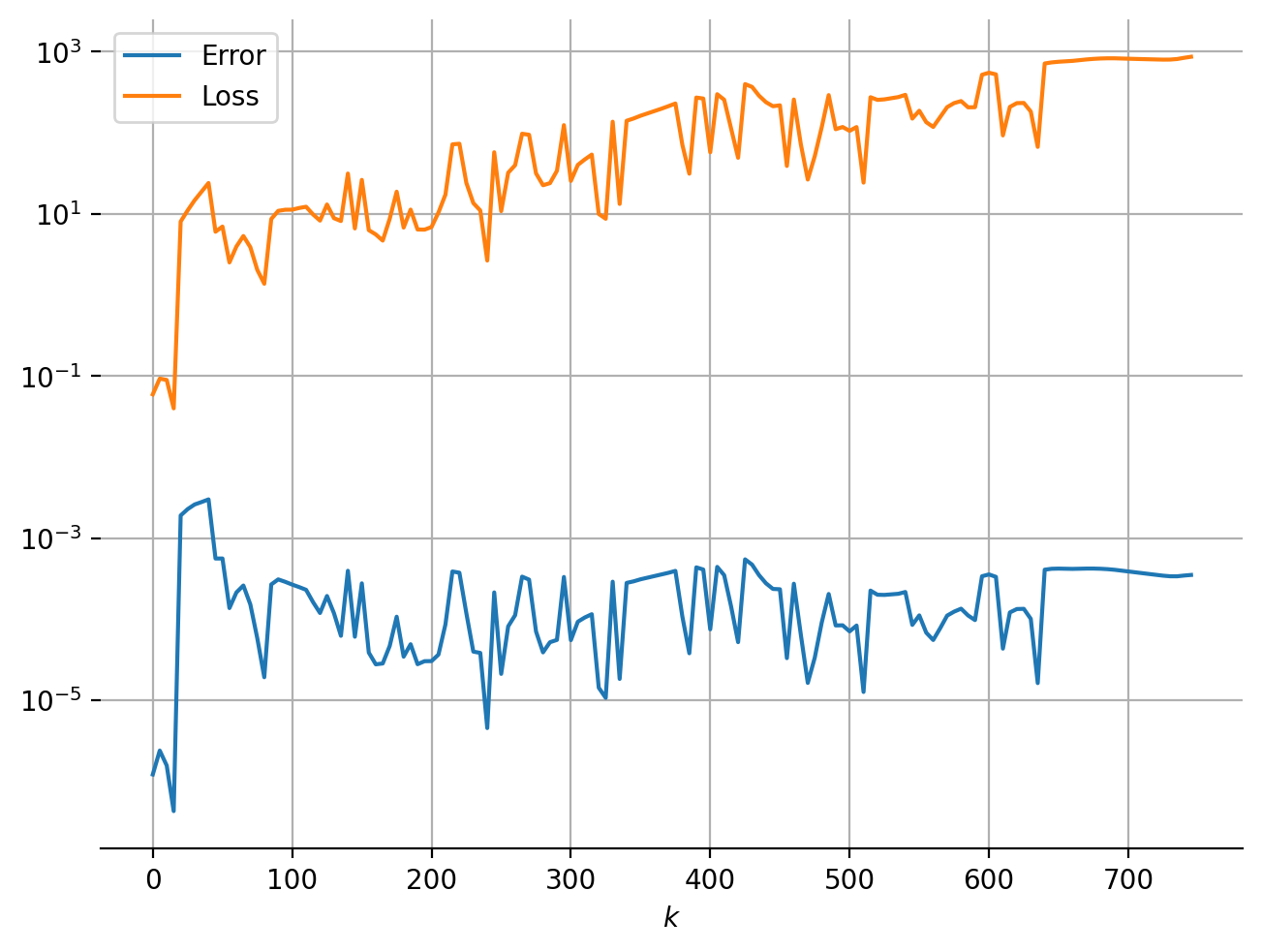}
	\caption{{\sf Loss} and {\sf Error} after 500 epochs as $k$ increases\; ($\varepsilon=1$, $\lambda=7$)}
         \label{fig:train}
     \end{subfigure}
     \hfill
     \begin{subfigure}[b]{0.43\textwidth}
         \centering
         \includegraphics[width=\textwidth]{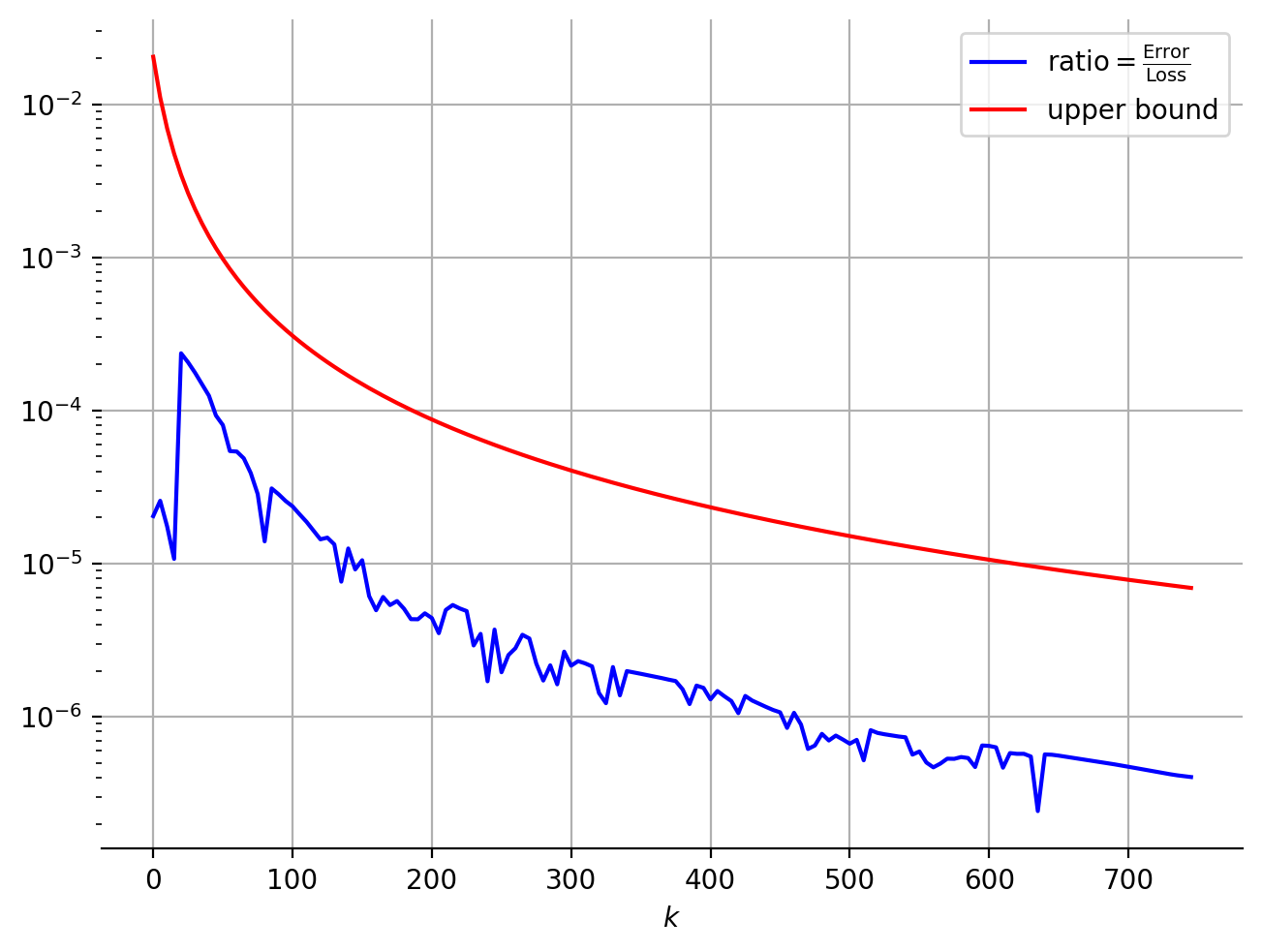}
	\caption{Ratio  ($\frac{\sf Error}{\sf Loss}$)  bounded by the upper bound $\frac{1}{\gamma^2} = \frac{4}{(k+14)^2}$ as $k$ increases \;($\varepsilon=1$, $\lambda=7$)}
         \label{fig:result}
     \end{subfigure}
     \caption{As $k$ increases, {\sf Loss} increases but {\sf Error} remains robust (Example \ref{example2nd})}
	\label{figfinaskincre}
\end{figure}
\end{itemize}
In Figure \ref{figfinaskincre}(a), we can observe that as $\lambda$ increases, {\sf Loss} also increases but {\sf Error} remains stable. This phenomenon is well explained by Figure \ref{figfinaskincre}(b) where the ratio ($\frac{\sf Error}{\sf Loss}$) rapidly decreases with the upper bound $\frac{1}{\gamma^2}=\frac{4}{(k+14)^2}$ as $k$ increases.
\centerline{}	
\centerline{}
\begin{itemize}
\item[(3)]
Let $k=10$ and $\lambda=15$.
\begin{figure}[!h]
     \centering
     \begin{subfigure}[b]{0.43\textwidth}
         \centering
         \includegraphics[width=\textwidth]{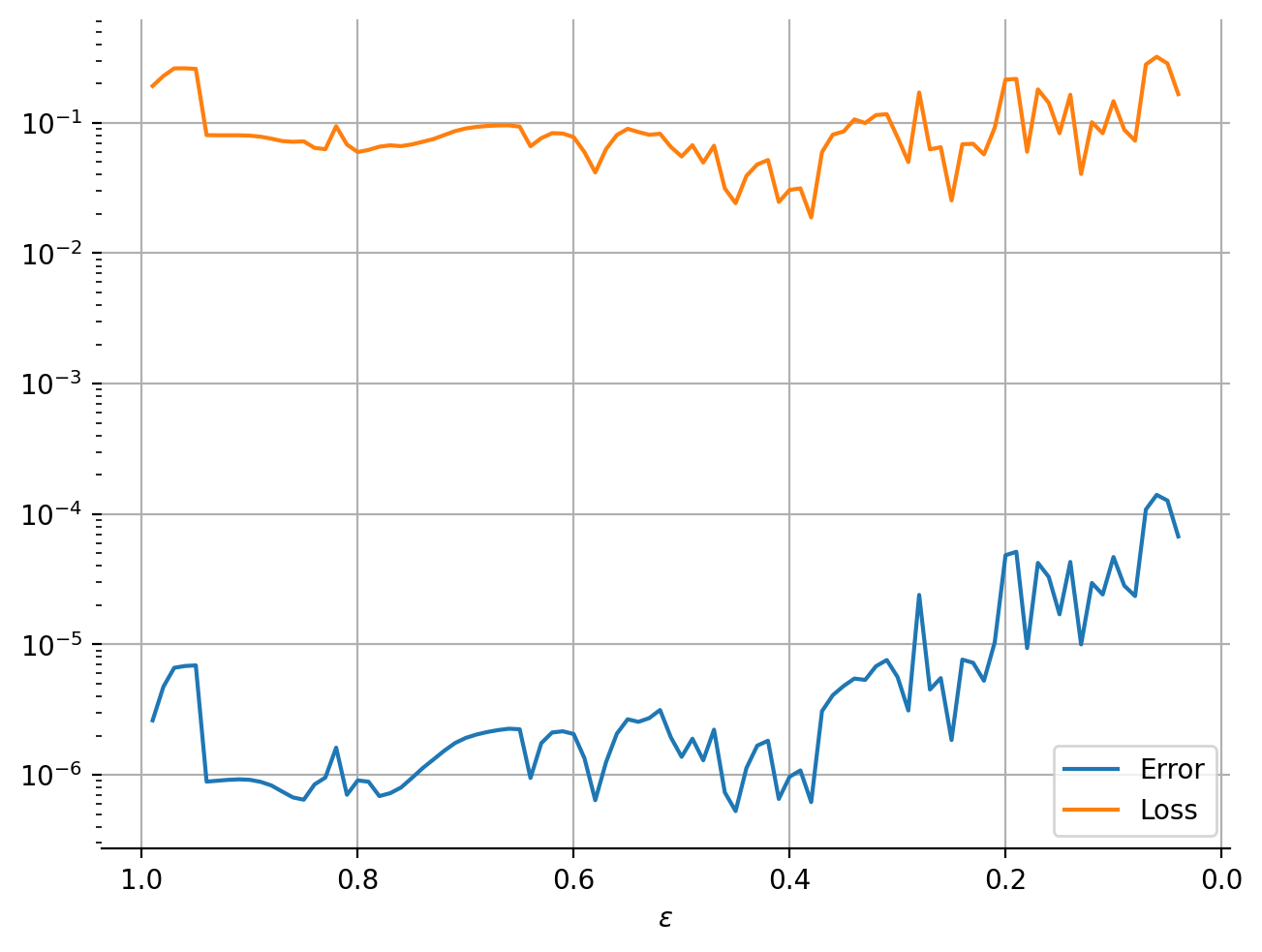}
	\caption{{\sf Loss} and {\sf Error} after 500 epochs as $\varepsilon$ decreases \; ($k=10$, $\lambda=15$)}
         \label{fig:train}
     \end{subfigure}
     \hfill
     \begin{subfigure}[b]{0.43\textwidth}
         \centering
         \includegraphics[width=\textwidth]{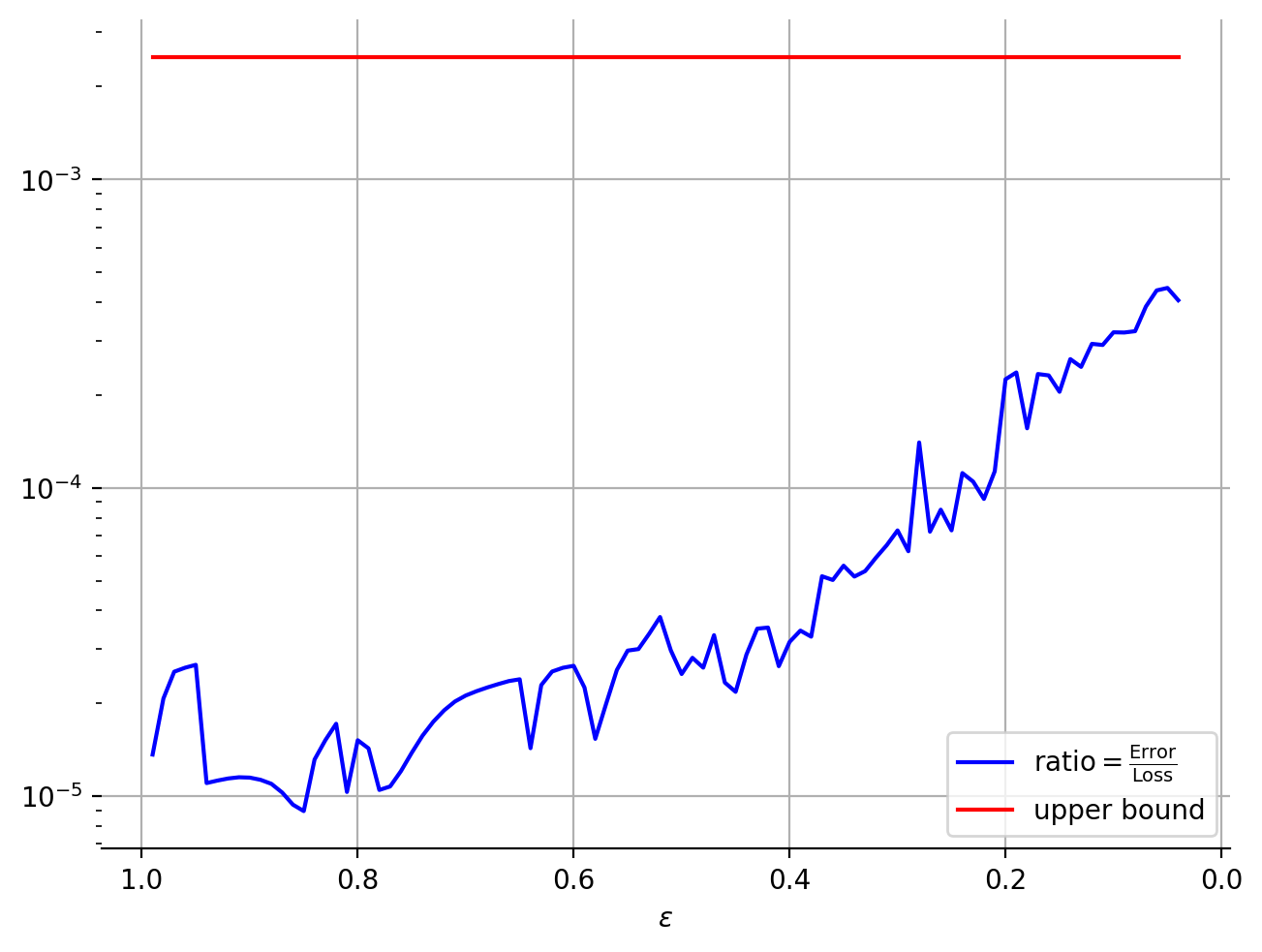}
	\caption{Ratio ($\frac{\sf Error}{\sf Loss}$) bounded by the upper bound $\frac{1}{\gamma^2}=\frac{1}{400}$ as $\varepsilon$ decreases \; ($k=10$, $\lambda=15$)}
         \label{fig:result}
     \end{subfigure}
    \caption{As $\varepsilon$ decreases, {\sf Loss} and {\sf Error} remain stable (Example \ref{example2nd})}
	\label{figvarepfin}
\end{figure}
\end{itemize}
	In Figure \ref{figvarepfin}(a), we can observe that as $\varepsilon$ decreases, {\sf Loss} and {\sf Error} remain stable. Since $\varepsilon$ is a very small value, the stability of {\sf Loss} is well-explained. Since $k$ and $\lambda$ are fixed as $10$ and $15$, respectively and $\varepsilon$ only decreases, we can expect that by our error estimates \eqref{errorestindis}, $\frac{\sf Error}{\sf Loss}$ does not change dramatically. This expectation is realized by Figure \ref{figvarepfin}(b) where $\frac{\sf Error}{\sf Loss}$ are bounded by $\frac{1}{\gamma^2}=\frac{1}{400}$, and hence the robustness of {\sf Error} is well-explained. 
\end{exam}
\centerline{}
\section{Conclusions and discussion} \label{conout}
\noindent
Based on the universal approximation theorem derived from the Stone-Weierstrass theorem, solutions to differential equations can be uniformly approximated by neural networks. However, due to the abstract nature of the Stone-Weierstrass theorem, finding explicit neural networks that approximate a solution to a differential equation would be very challenging. As an alternative, through a posteriori analysis, we have demonstrated that {\sf Error} is controlled by {\sf Loss}
(see \eqref{pinn1erbdd}, \eqref{engpinn2erbdd}, \eqref{weighterbdd}, \eqref{errorestindis}), and hence we can obtain effective error bounds in terms of {\sf Loss}. \\
Our analysis is specifically based on the Sobolev space theory with a variational approach. Initially, we rigorously established the existence and uniqueness of solutions to one-dimensional boundary value problems for second-order linear elliptic equations with highly irregular coefficients. There are two approaches for selecting a trial function to approximate the solution to a differential equation (see Section \ref{basicexperime}): the first approach is to directly use a neural network itself as a trial function ({\sf PINN I}), and the second approach is to use a function based on the neural network that exactly fulfills the given boundary conditions {\sf (PINN II)}. Experimentally, we confirmed that {\sf (PINN II)} is more efficient in approximating the solution. \\
Ultimately, our main error estimate is \eqref{errorestindis} which demonstrates that under the assumption of \eqref{maincondiforour},  $\frac{\sf Error}{\sf Loss}$ decreases rapidly with the upper bound $\frac{1}{\gamma^2}$ as $\gamma$ increases, so that {\sf Error} is robustly maintained even though {\sf Loss} increases. These results refute the conventional belief that numerical cost increases as the quantities of the coefficients increase. Our research indicates that PINN can be efficient numerical solvers for differential equations with large quantities of coefficients.\\
In the case of the upper bound of the ratio, Figure \ref{figdxcontraes}(b) and Figure \ref{figlastlambda}(b) describe very sharp upper bounds for $\frac{\sf Error}{\sf Loss}$, while Figure \ref{figaskincrero}(b) and Figure \ref{figfinaskincre}(b) present less sharp bounds, and hence further studies can be conducted to derive sharper upper bounds. Further discussion is required to determine under what additional conditions the rapid decrease in $\frac{\sf Error}{\sf Loss}$ can be expected. Empirically, it is observed that the larger quantities of the first-order coefficients $b$ and the zero-order coefficients $c$ imply the small values of $\frac{\sf Error}{\sf Loss}$. In future research, a rigorous mathematical demonstration is required to present additional explicit conditions on the coefficients that guarantee a rapid decrease in $\frac{\sf Error}{\sf Loss}$ beyond the condition \eqref{maincondiforour}.\\
The most innovative aspect of this paper is that we specifically calculated the constant arising in the $L^2$-estimate and, in particular, verified through $L^2$-contraction estimates that this constant decreases rapidly as the lower bound of the zero-order term increases. This is highly significant as it provides a concrete upper bound for the error-to-loss ratio in a posteriori error analysis of PINN, which enables us to estimate $L^2$-error by calculating $L^2$-training loss. Our research emphasizes the importance of precisely calculating constants in $L^2$-estimates in future studies. However, the current results are limited to the one-dimensional Dirichlet boundary value problem, and the condition on the coefficients for obtaining $L^2$-contraction estimates is somewhat restrictive because we need a constant $\gamma>0$ that is not small and satisfies $-\frac{1}{2}b'+c \geq \gamma$ in $I$. Further research is necessary to overcome these limitations.\\
Through this paper, it was confirmed that robust error estimates are guaranteed regardless of the quantities of coefficients in one-dimensional boundary value problems for linear elliptic equations. Moreover, we could confirm the possibility of an efficient numerical solution through PINN for differential equations with very large quantities of coefficients, and we expect that this will be applied to more general multi-dimensional elliptic and parabolic differential equations with singular coefficients and general boundary conditions.
\\
\centerline{}
The code used for numerical experiments in this paper is available at  \\
{\sf https://github.com/hahmYoo/Robust-error-estimates-of-PINN}.
\centerline{}
\centerline{}

\section*{Acknowledgments}
\noindent
We would like to thank the anonymous reviewers for their valuable comments and suggestions, which improve the content of this paper.

Jihahm Yoo \\
Korea Science Academy of KAIST, \\ 
Busan 47162, Republic of Korea, \\
E-mail: jihahmyoo@gmail.com
\centerline{}
\centerline{}
Haesung Lee\\
Department of Mathematics and Big Data Science,  \\
Kumoh National Institute of Technology, \\
Gumi, Gyeongsangbuk-do 39177, Republic of Korea, \\
E-mail: fthslt@kumoh.ac.kr, \; fthslt14@gmail.com
\end{document}